\documentclass[fleqn,pdftex]{article}

\usepackage{amsmath,amssymb,amsthm,url}
\usepackage{tikz}

\title{{\Large A note on derivability conditions}}
\author{Taishi Kurahashi}
\date{}

\theoremstyle{plain}
\newtheorem{thm}{Theorem}[section]
\newtheorem{lem}[thm]{Lemma}
\newtheorem{prop}[thm]{Proposition}
\newtheorem{cor}[thm]{Corollary}
\newtheorem{fact}[thm]{Fact}
\newtheorem{prob}[thm]{Problem}

\theoremstyle{definition}
\newtheorem{defn}[thm]{Definition}

\newtheorem{rem}[thm]{Remark}

\newcommand{\PA}{\mathsf{PA}}
\newcommand{\PR}{\mathrm{Pr}}
\newcommand{\PRL}{\mathrm{Pr}_\emptyset}
\newcommand{\Prf}{\mathrm{Prf}}
\newcommand{\Con}{\mathsf{Con}_\Phi}
\newcommand{\Fml}{\mathsf{Fml}}
\newcommand{\Sent}{\mathsf{Sent}}
\newcommand{\Even}{\mathsf{Even}}
\newcommand{\gdl}[1]{\ulcorner#1\urcorner}
\newcommand{\gn}{\mathsf{gn}}
\newcommand{\num}{\mathsf{num}}
\newcommand{\CS}{\mathsf{CS}}
\newcommand{\HB}[1]{\mathbf{HB#1}}
\newcommand{\Th}{\mathrm{Th}}

\newcommand{\D}[1]{\mathbf{D#1}}
\newcommand{\BD}[1]{\mathbf{B_{#1}}}
\newcommand{\DC}{\mathbf{\Delta_0 C}}
\newcommand{\SC}{\mathbf{\Sigma_1 C}}
\newcommand{\GC}{\mathbf{\Gamma C}}
\newcommand{\PC}{\mathbf{PC}}

\newcommand{\DU}[1]{\mathbf{D#1^U}}
\newcommand{\BDU}[1]{\mathbf{B_{#1}^U}}
\newcommand{\DCU}{\mathbf{\Delta_0 C^U}}
\newcommand{\SCU}{\mathbf{\Sigma_1 C^U}}
\newcommand{\GCU}{\mathbf{\Gamma C^U}}
\newcommand{\CB}{\mathbf{CB}}
\newcommand{\PCU}{\mathbf{PC^U}}

\newcommand{\DG}[1]{\mathbf{D#1^G}}
\newcommand{\DCG}{\mathbf{\Delta_0 C^G}}
\newcommand{\SCG}{\mathbf{\Sigma_1 C^G}}
\newcommand{\GCG}{\mathbf{\Gamma C^G}}
\newcommand{\PCG}{\mathbf{PC^G}}
\newcommand{\Ax}{\mathbf{Ax}}

\begin{document}

\maketitle

\begin{abstract}
We investigate relationships between versions of derivability conditions for provability predicates. 
We show several implications and non-implications between the conditions, and we discuss unprovability of consistency statements induced by derivability conditions. 
First, we classify already known versions of the second incompleteness theorem, and exhibit some new sets of conditions which are sufficient for unprovability of Hilbert--Bernays' consistency  statement. 
Secondly, we improve Buchholz's schematic proof of provable $\Sigma_1$-completeness. 
Then among other things, we show that Hilbert--Bernays' conditions and L\"ob's conditions are mutually incomparable. 
We also show that neither Hilbert--Bernays' conditions nor L\"ob's conditions accomplish G\"odel's original statement of the second incompleteness theorem.
\end{abstract}

\section{Introduction}

In his famous paper \cite{Goed31}, G\"odel proved the second incompleteness theorem with only a sketched proof. 
G\"odel explained that by formalizing his proof of the first incompleteness theorem, the consistency statement $\exists x(\Fml(x) \land \neg \PR_T(x))$ saying ``there exists a $T$-unprovable formula'' cannot be proved in $T$ if $T$ is consistent. 
To carry out his idea, it is desirable that the formula $\PR_T(x)$ enjoys some natural properties as a formalization of the notion of $T$-provability. 
He wrote that a detailed proof would be presented in a forthcoming work, but such a paper was not published after all.

The first detailed proof of the second incompleteness theorem was presented in the second volume of \textit{Grundlagen der Mathematik} \cite{HB39} by Hilbert and Bernays. 
Especially they formulated a set of conditions for provability predicates which is sufficient for the second incompleteness theorem. 
Let $\PR_T(x)$ be some $\Sigma_1$ provability predicate of $T$. 
They proved that if $\PR_T(x)$ satisfies the following conditions $\HB{1}$, $\HB{2}$ and $\HB{3}$\footnote{More precisely, Hilbert--Bernays' conditions were originally stated on proof predicate $\mathfrak{B}(x, y)$ rather than on provability predicate $\PR_T(x)$. 
For instance, the original statement of $\HB{1}$ is: If a formula with the number $j$ is derived from a formula with the number $i$, then $\exists x \mathfrak{B}(x, i) \to \exists x \mathfrak{B}(x, j)$ is provable. }, then the consistency statement $\forall x (\Fml(x) \land \PR_T(x) \to \neg \PR_T(\dot{\neg} x))$ cannot be proved in $T$ if $T$ is consistent. 
\begin{description}
	\item [$\HB{1}$] If $T \vdash \varphi \to \psi$, then $T \vdash \PR_T(\gdl{\varphi}) \to \PR_T(\gdl{\psi})$. 
	\item [$\HB{2}$] $T \vdash \PR_T(\gdl{\neg \varphi(x)}) \to \PR_T(\gdl{\neg \varphi(\dot{x})})$. 
	\item [$\HB{3}$] $T \vdash f(x) = 0 \to \PR_T(\gdl{f(\dot{x}) = 0})$ for every primitive recursive term $f(x)$.
\end{description}
Here $\gdl{\varphi(\dot{x})}$ is a primitive recursive term corresponding to a function calculating the G\"odel number of the formula $\varphi(\overline{n})$ from $n$, where $\overline{n}$ is the numeral for $n$. 
These conditions are called the \textit{Hilbert--Bernays derivability conditions}. 

L\"ob \cite{Lob55} proved that if $\PR_T(x)$ satisfies the following conditions $\D{1}$, $\D{2}$ and $\D{3}$, then L\"ob's theorem holds, that is, for any formula $\varphi$, if $T \vdash \PR_T(\gdl{\varphi}) \to \varphi$, then $T \vdash \varphi$. 

\begin{description}
	\item [$\D{1}$] If $T \vdash \varphi$, then $T \vdash \PR_T(\gdl{\varphi})$. 
	\item [$\D{2}$] $T \vdash \PR_T(\gdl{\varphi \to \psi}) \to (\PR_T(\gdl{\varphi}) \to \PR_T(\gdl{\psi}))$. 
	\item [$\D{3}$] $T \vdash \PR_T(\gdl{\varphi}) \to \PR_T(\gdl{\PR_T(\gdl{\varphi})})$. 
\end{description}

Note that every provability predicate automatically satisfies $\D{1}$. 
The conditions $\D{1}$ and $\D{2}$ were established by Hilbert and Bernays, and the condition $\D{3}$ was introduced by L\"ob. 
The conditions $\D{1}$, $\D{2}$ and $\D{3}$ are nowadays called the \textit{Hilbert--Bernays--L\"ob derivability conditions} which are well-known as sufficient conditions for a proof of the second incompleteness theorem. 
In fact, if $T$ is consistent, then the unprovability of the consistency statement $\neg \PR_T(\gdl{0 \neq 0})$ in $T$ is an immediate corollary of L\"ob's theorem. 
The Hilbert--Bernays--L\"ob derivability conditions together with L\"ob's theorem are basis for modal logical investigations of provability predicates (see \cite{AB05,Boo93,JD98,Smo85}). 

Other sufficient conditions for the second incompleteness theorem were formulated by authors such as Jeroslow, Montagna and Buchholz. 
Jeroslow \cite{Jer73} proved that the following condition which is a variant of $\D{3}$ implies the unprovability of $\forall x (\Fml(x) \land \PR_T(x) \to \neg \PR_T(\dot{\neg} x))$. 
\begin{itemize}
	\item $T \vdash \PR_T(t) \to \PR_T(\gdl{\PR_T(t)})$ for every primitive recursive term $t$.
\end{itemize}
Notice that $\D{3}$ and Jeroslow's condition are instances of the following provable $\Sigma_1$-completeness because $\PR_T(x)$ is $\Sigma_1$. 
\begin{description}
	\item [$\SC$] If $\varphi$ is a $\Sigma_1$ sentence, then $T \vdash \varphi \to \PR_T(\gdl{\varphi})$. 
\end{description}

Montagna \cite{Mon79} proved that the following two conditions are sufficient for L\"ob's theorem. 
\begin{itemize}
	\item $T \vdash \forall x($``$x$ is a logical axiom'' $\to \PR_T(x))$. 
	\item $T \vdash \forall x \forall y(\Fml(x) \land \Fml(y) \to (\PR_T(x \dot{\to} y) \to (\PR_T(x) \to \PR_T(y))))$. 
\end{itemize}
By Montagna's argument, we can conclude that these two conditions imply the unprovability of $\exists x(\Fml(x) \land \neg \PR_T(x))$. 

At last, in Buchholz's lecture note \cite{Buc93}, the following condition was introduced and it was proved that this condition implies $\D{2}$ and $\SC$. 
\begin{itemize}
	\item For all $m \geq 1$, \\
	if $T \vdash \forall \vec{x}(\varphi_1(\vec{x}) \to (\varphi_2(\vec{x}) \to (\cdots \to (\varphi_{m-1}(\vec{x}) \to \varphi_m(\vec{x})) \cdots)))$, \\
	then $T \vdash \forall \vec{x} (\PR_T(\gdl{\varphi_1(\vec{\dot{x}})}) \to (\PR_T(\gdl{\varphi_2(\vec{\dot{x}})}) \to \\
	\ \ \ \ \ \ \ \ \ \ \ \ \  (\cdots \to (\PR_T(\gdl{\varphi_{m-1}(\vec{\dot{x}})}) \to \PR_T(\gdl{\varphi_m(\vec{\dot{x}})})) \cdots)))$. 
\end{itemize}
Thus Buchholz's condition implies the unprovability of $\neg \PR_T(\gdl{0 \neq 0})$. 

Roughly speaking, every set of derivability conditions introduced above is sufficient for unprovability of consistency statements, but such a rough understanding does not allow us to grasp the situation of the second incompleteness theorem accurately. 
Strictly speaking, these sets of sufficient conditions do not induce the same consequence because there are three different consistency statements $\mathsf{Con}^H \equiv \forall x (\Fml(x) \land \PR_T(x) \to \neg \PR_T(\dot{\neg} x))$, $\mathsf{Con}^L \equiv \neg \PR_T(\gdl{0 \neq 0})$ and $\mathsf{Con}^G \equiv \exists x(\Fml(x) \land \neg \PR_T(x))$ in our context, and each of these sets of conditions implies the unprovability of one of these consistency statements. 
Here superscripts `H', `L' and `G' stand for Hilbert--Bernays, L\"ob and G\"odel, respectively. 
It is easy to see that $\mathsf{Con}^H$ implies $\mathsf{Con}^L$, and $\mathsf{Con}^L$ implies $\mathsf{Con}^G$. 
However the converse implications do not hold in general. 

In order to clarify the situation of several versions of derivability conditions, in this paper, we investigate relationships between the conditions. 
The following figure shows the situation for implications between prominent sets of conditions for $\Sigma_1$ formulas satisfying $\D{1}$. 

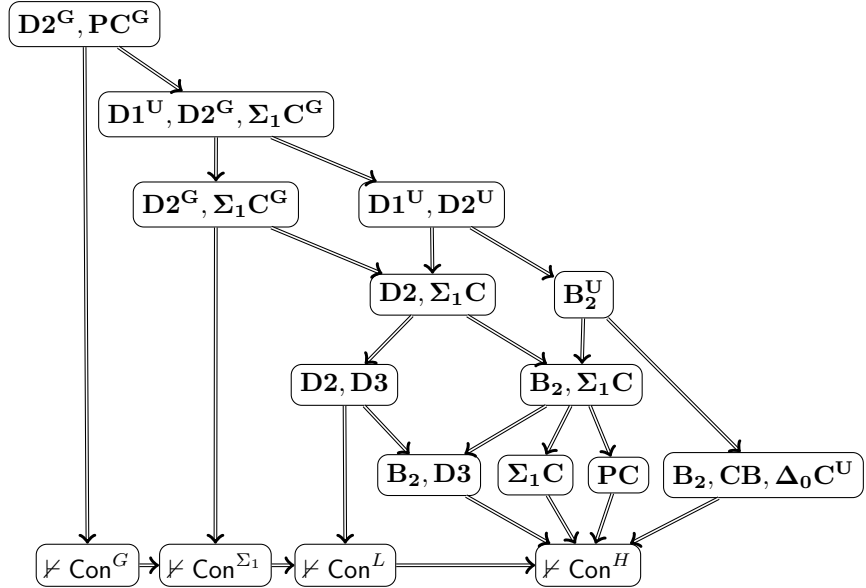
\vspace{0.2in}

\begin{tikzpicture}
\tikzset{dc/.style={draw, rounded corners}};

	\node[dc] (Con2) at (0.36, 0) [right] {$\nvdash \mathsf{Con}^G$};
	\node[dc] (ConS) at (2, 0) [right] {$\nvdash \mathsf{Con}^{\Sigma_1}$};
	\node[dc] (Con1) at (3.8, 0) [right] {$\nvdash \mathsf{Con}^L$};
	\node[dc] (Con0) at (7, 0) [right] {$\nvdash \mathsf{Con}^H$};
	\node[dc] (G2-2) at (4.9, 1.2) [right] {$\BD{2}, \D{3}$};
	\node[dc] (Jeroslow) at (6.5, 1.2) [right] {$\SC$}; 
	\node[dc] (G2-3) at (7.7, 1.2) [right] {$\PC$}; 
	\node[dc] (HB) at (8.7, 1.2) [right] {$\BD{2}, \CB, \DCU$}; 
	\node[dc] (Lob) at (3.75, 2.4) [right] {$\D{2}, \D{3}$}; 
	\node[dc] (BS) at (6.8, 2.4) [right] {$\BD{2}, \SC$}; 
	\node[dc] (HBL) at (4.8, 3.6) [right] {$\D{2}, \SC$}; 
	\node[dc] (Kurahashi) at (7.25, 3.6) [right] {$\BDU{2}$}; 
	\node[dc] (Buchholz) at (4.65, 4.8) [right] {$\DU{1}, \DU{2}$}; 
	\node[dc] (Global1) at (1.19, 6) [right] {$\DU{1}, \DG{2}, \SCG$};
	\node[dc] (Global2) at (1.65, 4.8) [right] {$\DG{2}, \SCG$};
	\node[dc] (Montagna) at (0, 7.2) [right] {$\DG{2}, \PCG$}; 

 	\draw[double,->] (Con2) -- (ConS);
 	\draw[double,->] (ConS) -- (Con1);
	\draw[double,->] (Con1) -- (Con0);
	\draw[double,->] (G2-2) -- (Con0);
	\draw[double,->] (Jeroslow) -- (Con0);
	\draw[double,->] (G2-3) -- (Con0);
	\draw[double,->] (HB) -- (Con0);
	\draw[double,->] (Lob) -- (Con1);
	\draw[double,->] (Lob) -- (G2-2);
	\draw[double,->] (BS) -- (G2-2);
	\draw[double,->] (BS) -- (Jeroslow);
	\draw[double,->] (BS) -- (G2-3);
	\draw[double,->] (HBL) -- (Lob);
	\draw[double,->] (HBL) -- (BS);
	\draw[double,->] (Kurahashi) -- (BS);
	\draw[double,->] (Kurahashi) -- (HB);
	\draw[double,->] (Buchholz) -- (HBL);
	\draw[double,->] (Buchholz) -- (Kurahashi);
	\draw[double,->] (Montagna) -- (Global1);
	\draw[double,->] (Global1) -- (Buchholz);
	\draw[double,->] (Global1) -- (Global2);
	\draw[double,->] (Global2) -- (HBL);
	\draw[double,->] (Global2) -- (ConS);
	\draw[double,->] (Montagna) -- (Con2);

\end{tikzpicture}

\vspace{0.2in}

In Section \ref{Sec:DC}, we introduce and investigate versions of derivability conditions. 
Each of these conditions is classified as one of three versions of derivability conditions, namely, local version, uniform version and global version. 
Among other things, we show that each of two new sets $\{\D{1}, \BD{2}, \D{3}\}$ and $\{\D{1}, \PC\}$ of derivability conditions is sufficient for the unprovability of the consistency statement $\mathsf{Con}^H$ (see the next section for precise definitions of these conditions). 
Then currently we know that four sets $\{\BD{2}, \CB, \DCU\}$, $\{\D{1}, \BD{2}, \D{3}\}$, $\{\D{1}, \SC\}$ and $\{\D{1}, \PC\}$ are sufficient for $T \nvdash \mathsf{Con}^H$, the set $\{\D{1}, \D{2}, \D{3}\}$ (L\"ob's conditions) is sufficient for $T \nvdash \mathsf{Con}^L$, and the set $\{\D{1}, \DG{2}, \PCG\}$ is sufficient for $T \nvdash \mathsf{Con}^G$. 
Here $\{\BD{2}, \CB, \DCU\}$, $\{\D{1}, \SC\}$ and $\{\D{1}, \DG{2}, \PCG\}$ correspond to Hilbert and Bernays' conditions, Jeroslow's conditions and Montagna's conditions, respectively. 

In Section \ref{Sec:MT}, we improve Buchholz's proof of provable $\Sigma_1$-completeness $\SC$. 
More precisely, we prove that if $\PR_T(x)$ satisfies the following condition $\BDU{2}$ which is precisely the $m=2$ case of Buchholz's condition, then the uniform version of $\SC$ holds.
\begin{description}
	\item [$\BDU{2}$] If $\displaystyle T \vdash \forall \vec{x}\left(\varphi(\vec{x}) \to \psi(\vec{x}) \right)$, then $\displaystyle T \vdash \forall \vec{x} (\PR_T(\gdl{\varphi(\vec{\dot{x}})}) \to \PR_T(\gdl{\psi(\vec{\dot{x}})}))$. 
\end{description}

In Section \ref{Sec:W}, we give some examples of formulas, and from these examples, several non-implications between conditions are obtained. 
For instance, from our examples, we obtain that $\{\BD{2}, \CB, \DCU\}$, $\{\D{1}, \BD{2}, \D{3}\}$, $\{\D{1}, \SC\}$ and $\{\D{1}, \PC\}$ are pairwise incomparable, and each of them is not sufficient for $T \nvdash \mathsf{Con}^L$. 
Also we obtain that $\{\D{1}, \D{2}, \D{3}\}$ is not comparable with each of $\{\BD{2}, \CB, \DCU\}$, $\{\D{1}, \SC\}$ and $\{\D{1}, \PC\}$, and it is not sufficient for $T \nvdash \mathsf{Con}^G$. 
Furthermore, we show that even stronger set $\{\DU{1}, \DG{2}, \SCG\}$ is not sufficient for $T \nvdash \mathsf{Con}^G$. 
From the last observation, we can say that both of the Hilbert--Bernays derivability conditions and the Hilbert--Bernays--L\"ob derivability conditions do not accomplish G\"odel's original statement of the second incompleteness theorem.

\section{Derivability conditions}\label{Sec:DC}

Throughout this paper, $S$ and $T$ denote recursively axiomatized consistent extensions of Peano Arithmetic $\PA$ in the language of first-order arithmetic. 
The theory $S$ is intended as a metatheory, and we assume that $T$ is an extension of $S$. 
Let $\mathcal{L}_A$ be the language of arithmetic including $\{0, \mathsf{s}, +, \times\}$, and we can freely use terms corresponding to some primitive recursive functions. 
The numeral $\overline{n}$ for a natural number $n$ is the closed term $\underbrace{\mathsf{s}(\mathsf{s}(\cdots \mathsf{s}}_{n\ \text{times}}(0) \cdots ))$. 
This explicit form of numerals is used in Section \ref{Sec:MT}. 
We fix some natural G\"odel numbering, and for each $\mathcal{L}_A$-formula $\varphi$, let $\gdl{\varphi}$ be the numeral for the G\"odel number of $\varphi$. 
Let $x \dot{\to} y$ and $\dot{\neg}x$ denote primitive recursive terms such that for any formulas $\varphi$ and $\psi$, $\PA \vdash \gdl{\varphi} \dot{\to} \gdl{\psi} = \gdl{\varphi \to \psi}$ and $\PA \vdash \neg \gdl{\varphi} = \gdl{\neg \varphi}$. 

Let $\Delta_0 = \Sigma_0 = \Pi_0$ be the set of all formulas whose quantifiers are all bounded. 
Let $\Sigma_{n+1}$ and $\Pi_{n+1}$ ($n \geq 0$) be the least sets of formulas satisfying the following conditions: 
\begin{enumerate}
	\item $\Sigma_n \cup \Pi_n \subseteq \Sigma_{n+1} \cap \Pi_{n+1}$; 
	\item $\Sigma_{n+1}$ (resp.~$\Pi_{n+1}$) is closed under conjunction, disjunction, bounded quantification, and existential (resp.~universal) quantification; 
	\item If $\varphi$ is in $\Sigma_{n+1}$ (resp.~$\Pi_{n+1}$), then $\neg \varphi$ is in $\Pi_{n+1}$ (resp.~$\Sigma_{n+1}$); 
	\item If $\varphi$ is in $\Sigma_{n+1}$ (resp.~$\Pi_{n+1}$) and $\psi$ is in $\Pi_{n+1}$ (resp.~$\Sigma_{n+1}$), then $\varphi \to \psi$ is in $\Pi_{n+1}$ (resp.~$\Sigma_{n+1}$). 
\end{enumerate}
Throughout this paper, $\Gamma$ denotes $\Sigma_n$ or $\Pi_n$ for some $n \geq 0$. 
We say a formula $\varphi$ is $\Gamma$ if $\varphi \in \Gamma$. 
A formula $\varphi$ is said to be $\Delta_1$ if it is provably equivalent to both some $\Sigma_1$ formula and some $\Pi_1$ formula in $\PA$. 
Let $\Fml(x)$, $\Sent(x)$ and $\Sigma_z(x)$ be $\Delta_1$ formulas saying that ``$x$ is the G\"odel number of an $\mathcal{L}_A$-formula'', ``$x$ is the G\"odel number of an $\mathcal{L}_A$-sentence'' and ``$x$ is the G\"odel number of a $\Sigma_z$ formula'', respectively. 
We assume that $\PA$ can derive natural facts about these formulas such as $\forall z \exists x > z \Fml(x)$. 

We say a formula $\PR(x)$ is a \textit{provability predicate} of a theory $U$ (in $\PA$) if it weakly represents the set of all theorems of $U$ in $\PA$, that is, for any natural number $n$, $\PA \vdash \PR(\overline{n})$ if and only if $n$ is the G\"odel number of some theorem of $U$. 
Also we say a formula $\tau(v)$ is a \textit{numeration} of $U$ (in $\PA$) if it weakly represents the set of all axioms of $U$ in $\PA$, that is, for any natural number $n$, $\PA \vdash \tau(\overline{n})$ if and only if $n$ is the G\"odel number of some axiom of $U$. 
For each numeration $\tau(v)$ of $U$, we can naturally construct a formula $\Prf_\tau(x, y)$ saying that ``$y$ is the code of a proof of a formula with the G\"odel number $x$ from the set of all sentences satisfying $\tau(v)$'' (see Feferman \cite{Fef60}). 
We may assume $\PA \vdash \forall x \forall y(\Prf_\tau(x, y) \to x \leq y)$. 
If $\tau(v)$ is a $\Sigma_n$ numeration of $U$ for $n > 0$, then the formula $\PR_\tau(x) : \equiv \exists y \Prf_\tau(x, y)$ is a $\Sigma_n$ provability predicate of $U$. 
If it is not necessary to specify a particular numeration of $U$, $\Prf_U(x, y)$ and $\PR_U(x)$ denote $\Prf_\tau(x, y)$ and $\PR_\tau(x)$ for some fixed numeration $\tau(v)$ of $U$, respectively. 

For each finitely axiomatized theory $T_0$, let $[T_0](x)$ be the formula $\bigvee_{\varphi \in T_0} (x = \gdl{\varphi})$. 
Then $[T_0](x)$ is a numeration of $T_0$. 
Let $\bigwedge T_0$ be the conjunction of all axioms of $T_0$, and let $\PRL(x)$ be a natural provability predicate of first-order predicate calculus in the language $\mathcal{L}_A$. 
Then the following lemma holds (see Feferman \cite{Fef60}). 

\begin{lem}[Formalized deduction theorem]\label{FDT}
For any finitely axiomatized theory $T_0$, $\PA \vdash \forall x(\PR_{[T_0]}(x) \leftrightarrow \PRL(\gdl{\bigwedge T_0} \dot{\to} x))$. 
\end{lem}

Throughout this paper, the formula $\Phi(x)$ is intended to denote some provability predicate of $T$. 
However, we deal with more general situations, that is, $\Phi(x)$ may not be any provability predicate of $T$. 
In this section, we introduce a lot of conditions for $\Phi(x)$ which are satisfied by naturally constructed provability predicates $\PR_T(x)$. 
The remainder of this section is separated into three subsections, and in each of these subsections, we introduce local derivability conditions, uniform derivability conditions and global derivability conditions, respectively. 

For each formula $\Phi(x)$, we define four kinds of consistency statements based on $\Phi(x)$. 

\begin{defn}\leavevmode
\begin{enumerate}
	\item $\Con^H : \equiv \forall x (\Fml(x) \land \Phi(x) \to \neg \Phi(\dot{\neg} x))$. 
	\item $\Con^L : \equiv \neg \Phi(\gdl{0 \neq 0})$. 
	\item $\Con^G : \equiv \exists x (\Fml(x) \land \neg \Phi(x))$. 
	\item $\Con^{\Sigma_1} : \equiv \exists x(\Sigma_1(x) \land \Sent(x) \land \neg \Phi(x))$. 
\end{enumerate}
\end{defn}

The first consistency statement $\Con^H$ is adopted in Hilbert and Bernays \cite{HB39} and Feferman \cite{Fef60}. 
The second sentence $\Con^L$ is the most tractable one, and it is widely used in the context of modal logical investigations of provability predicates. 
G\"odel \cite{Goed31} stated his second incompleteness theorem with the consistency statement $\Con^G$. 
The last consistency statement $\Con^{\Sigma_1}$ states that there exists a $T$-unprovable $\Sigma_1$ sentence.

\subsection{Local derivability conditions}

We introduce the weakest version of derivability conditions which are called local derivability conditions. 

\begin{defn}[Local derivability conditions]
\leavevmode
\begin{description} 
	\item [$\D{1}$] If $T \vdash \varphi$, then $S \vdash \Phi(\gdl{\varphi})$ for any formula $\varphi$. 
	\item [$\D{2}$] $S \vdash \Phi(\gdl{\varphi \to \psi}) \to (\Phi(\gdl{\varphi}) \to \Phi(\gdl{\psi}))$ for any formulas $\varphi$ and $\psi$. 
	\item [$\D{3}$] $S \vdash \Phi(\gdl{\varphi}) \to \Phi(\gdl{\Phi(\gdl{\varphi})})$ for any formula $\varphi$. 
	\item [$\GC$] $S \vdash \varphi \to \Phi(\gdl{\varphi})$ for any $\Gamma$ sentence $\varphi$. 
	\item [$\BD{m}$ ($m \geq 1$)] If $\displaystyle T \vdash \bigwedge_{0 < i < m} \varphi_i \to \varphi_m$, then $\displaystyle S \vdash \bigwedge_{0 < i < m} \Phi(\gdl{\varphi_i}) \to \Phi(\gdl{\varphi_m})$ for any formulas $\varphi_1, \ldots, \varphi_m$. 
	\item [$\PC$] $S \vdash \PRL(\gdl{\varphi}) \to \Phi(\gdl{\varphi})$ for any formula $\varphi$. 
\end{description}
\end{defn}

The condition $\D{1}$ is automatically satisfied by all provability predicates of $T$. 
The conditions $\D{2}$, $\D{3}$ and $\SC$ were introduced by Hilbert and Bernays \cite{HB39}, L\"ob \cite{Lob55} and Feferman \cite{Fef60}, respectively. 
It is known that natural provability predicates $\PR_T(x)$ satisfy full local derivability conditions. 
In particular, Feferman proved $\SC$ for the provability predicate $\PR_\mathsf{Q}(x)$ of Robinson's arithmetic $\mathsf{Q}$ (cf.~\cite{TMR53}). 
The conditions $\BD{m}$ ($m \geq 1$) were introduced by Buchholz \cite{Buc93}. 
The condition $\BD{1}$ is precisely $\D{1}$, and the condition $\BD{2}$ is precisely the condition $\HB{1}$ described in the introduction. 
The condition $\BD{2}$ was also discussed by Montagna \cite{Mon79} and Visser \cite{Vis16}. 
The last condition $\PC$ says that $\Phi(x)$ contains predicate calculus. 

We prove the basic implications between local derivability conditions. 
For example, the first clause of the following proposition says that if a formula $\Phi(x)$ satisfies $\D{1}$, then it also satisfies $\DC$. 

\begin{prop}\label{LP1}\leavevmode
\begin{enumerate}
	\item $\D{1} \Rightarrow \DC$. 
	
	\item $\DC$ and $\BD{m}$ for some $m \geq 1 \Rightarrow \D{1}$. 

	\item $\BD{3} \Rightarrow \D{2}$. 

	\item The following are equivalent: 
	\begin{enumerate}
		\item $\D{1}$ and $\D{2}$. 
		\item $\BD{m}$ for all $m \geq 1$. 
		\item $\D{1}$ and $\BD{m}$ for some $m \geq 3$. 
		\item $\DC$ and $\BD{m}$ for some $m \geq 3$. 
	\end{enumerate}

	\item If $\Phi(x)$ is a $\Gamma$ formula, then $\GC \Rightarrow \D{3}$. 	
	
	\item $\BD{2}$ and $\PC \iff \BD{2}$ and $\SC$. 
	
	\item $\BD{2}$ and $\PC \Rightarrow \D{1}$. 
	
	\item $\D{1}$, $\D{2}$ and $\PC \iff \D{1}$, $\D{2}$ and $\SC$. 

\end{enumerate}
\end{prop}
\begin{proof}
1. Suppose $\Phi(x)$ satisfies $\D{1}$. 
Let $\varphi$ be any $\Delta_0$ sentence. 
Then $\varphi$ is decidable in $\PA$. 
If $\PA \vdash \varphi$, then $S \vdash \Phi(\gdl{\varphi})$ by $\D{1}$, and hence $S \vdash \varphi \to \Phi(\gdl{\varphi})$. 
If $\PA \vdash \neg \varphi$, then $S \vdash \varphi \to \Phi(\gdl{\varphi})$. 

2. Suppose $\Phi(x)$ satisfies $\DC$ and $\BD{m}$ for some $m \geq 1$. 
Let $\varphi$ be any formula with $T \vdash \varphi$. 
Then $T \vdash \underbrace{0=0 \land \cdots \land 0=0}_{m - 1} \to \varphi$. 
By $\BD{m}$, we have $S \vdash \Phi(\gdl{0=0}) \to \Phi(\gdl{\varphi})$. 
By $\DC$, $S \vdash 0 = 0 \to \Phi(\gdl{0=0})$, and hence $S \vdash \Phi(\gdl{0=0})$. 
We conclude $S \vdash \Phi(\gdl{\varphi})$. 

3. Since $T \vdash (\varphi \to \psi) \land \varphi \to \psi$, we obtain $S \vdash \Phi(\gdl{\varphi \to \psi}) \land \Phi(\gdl{\varphi}) \to \Phi(\gdl{\psi})$ by $\BD{3}$. 

4. $(a) \Rightarrow (b)$ is well-known in the context of modal logic. 
$(b) \Rightarrow (c)$ is trivial. 
$(c) \Leftrightarrow (d)$ follows from clauses 1 and 2. 
We prove $(c) \Rightarrow (a)$: 
Suppose $\Phi(x)$ satisfies $\D{1}$ and $\BD{m}$ for some $m \geq 3$. 
By clause 3, it suffices to prove that $\Phi(x)$ satisfies $\BD{3}$. 
Suppose $T \vdash \varphi_1 \land \varphi_2 \to \varphi_3$. 
Then $T \vdash \varphi_1 \land \varphi_2 \land \underbrace{0=0 \land \cdots \land 0=0}_{m - 3} \to \varphi_3$. 
By $\BD{m}$, we obtain $S \vdash \Phi(\gdl{\varphi_1}) \land \Phi(\gdl{\varphi_2}) \land \Phi(\gdl{0=0}) \to \Phi(\gdl{\varphi_3})$. 
By $\D{1}$, we have $S \vdash \Phi(\gdl{0=0})$. 
Hence $S \vdash \Phi(\gdl{\varphi_1}) \land \Phi(\gdl{\varphi_2}) \to \Phi(\gdl{\varphi_3})$. 

5. Trivial. 

6. $(\Rightarrow)$: 
Assume that $\Phi(x)$ satisfies $\BD{2}$ and $\PC$. 
Let $\varphi$ be any $\Sigma_1$ sentence. 
Let $T_0$ be some finite subtheory of $T$ containing Robinson's arithmetic $\mathsf{Q}$.  
By $\PC$, $S \vdash \PRL(\gdl{\bigwedge T_0 \to \varphi}) \to \Phi(\gdl{\bigwedge T_0 \to \varphi})$. 
Here $\PRL(\gdl{\bigwedge T_0 \to \varphi})$ is equivalent to $\PR_{[T_0]}(\gdl{\varphi})$ by formalized deduction theorem (Lemma \ref{FDT}), and therefore we obtain $S \vdash \PR_{[T_0]}(\gdl{\varphi}) \to \Phi(\gdl{\bigwedge T_0 \to \varphi})$. 
Since $T_0$ is a subtheory of $T$, we have $T \vdash (\bigwedge T_0 \to \varphi) \to \varphi$. 
By $\BD{2}$, $S \vdash \Phi(\gdl{\bigwedge T_0 \to \varphi}) \to \Phi(\gdl{\varphi})$. 
Thus we obtain $S \vdash \PR_{[T_0]}(\gdl{\varphi}) \to \Phi(\gdl{\varphi})$. 
Since $T_0$ contains $\mathsf{Q}$, $\SC$ holds for $\PR_{[T_0]}(x)$, and hence $S \vdash \varphi \to \PR_{[T_0]}(\gdl{\varphi})$. 
Therefore $S \vdash \varphi \to \Phi(\gdl{\varphi})$. 

$(\Leftarrow)$: 
Suppose $\Phi(x)$ satisfies $\BD{2}$ and $\SC$. 
Let $\varphi$ be any formula. 
Since $\PRL(\gdl{\varphi})$ is a $\Sigma_1$ sentence, $S \vdash \PRL(\gdl{\varphi}) \to \Phi(\gdl{\PRL(\gdl{\varphi})})$. 
Since $T$ is an extension of $\PA$, $T \vdash \PRL(\gdl{\varphi}) \to \varphi$ by the reflexiveness of $\PA$ (see \cite{Lin03}). 
By $\BD{2}$, $S \vdash \Phi(\gdl{\PRL(\gdl{\varphi})}) \to \Phi(\gdl{\varphi})$. 
Therefore $S \vdash \PRL(\gdl{\varphi}) \to \Phi(\gdl{\varphi})$.  

7. This follows from clauses 2 and 6. 

8. This equivalence follows from clauses 4 and 6. 
\end{proof}

Before describing several versions of the second incompleteness theorem, we prepare two propositions. 

\begin{prop}\label{LP2}\leavevmode
\begin{enumerate}
	\item If $\Phi(x)$ satisfies $\D{1}$, then $S \vdash \Con^H \to \Con^L$. 
	\item $\PA \vdash \Con^L \to \Con^{\Sigma_1}$. 
	\item $\PA \vdash \Con^{\Sigma_1} \to \Con^G$. 
\end{enumerate}
\end{prop}
\begin{proof}
1. Suppose $\Phi(x)$ satisfies $\D{1}$, then $S \vdash \Phi(\gdl{0=0})$. 
Since $\PA \vdash \Con^H \to (\Phi(\gdl{0=0}) \to \neg \Phi(\gdl{0 \neq 0}))$, we have $S \vdash \Con^H \to \Con^L$. 

Clauses 2 and 3 are obvious. 
\end{proof}

The following proposition is a part of G\"odel's first incompleteness theorem. 

\begin{prop}\label{LP3}
Let $\varphi$ be a sentence satisfying $\PA \vdash \varphi \leftrightarrow \neg \Phi(\gdl{\varphi})$. 
If $\Phi(x)$ satisfies $\D{1}$, then $T \nvdash \varphi$. 
\end{prop}
\begin{proof}
Suppose $\Phi(x)$ satisfies $\D{1}$. 
If $T \vdash \varphi$, then by $\D{1}$, $S \vdash \Phi(\gdl{\varphi})$. 
By the choice of $\varphi$, $S \vdash \neg \varphi$. 
This contradicts the consistency of $T$ because $T$ is an extension of $S$. 
Therefore $T \nvdash \varphi$. 
\end{proof}

It is well-known that for proofs of the second incompleteness theorem, the Hilbert--Bernays--L\"ob derivability conditions $\D{1}$, $\D{2}$ and $\D{3}$ are sufficient. 
This is essentially due to L\"ob (see \cite{Boo93,Lin03}). 

\begin{thm}[L\"ob \cite{Lob55}]\label{G2}\leavevmode
If $\Phi(x)$ satisfies $\D{1}$, $\D{2}$ and $\D{3}$, then $T \nvdash \Con^L$. 
\end{thm}

Notice that $\{\D{1}, \BD{2}, \D{3}\}$ is weaker than $\{\D{1}, \D{2}, \D{3}\}$ by Proposition \ref{LP1}.4. 
For the former conditions, we obtain another version of the second incompleteness theorem. 

\begin{thm}\label{G2-2}\leavevmode
If $\Phi(x)$ satisfies $\D{1}$, $\BD{2}$ and $\D{3}$, then $T \nvdash \Con^H$. 
\end{thm}
\begin{proof}
Suppose $\Phi(x)$ satisfies $\D{1}$, $\BD{2}$ and $\D{3}$. 
Let $\varphi$ be a sentence satisfying $\PA \vdash \varphi \leftrightarrow \neg \Phi(\gdl{\varphi})$. 
The existence of such a sentence $\varphi$ follows from the Fixed Point Lemma (see \cite{Lin03}). 
Since $T \vdash \Phi(\gdl{\varphi}) \to \neg \varphi$, we have $S \vdash \Phi(\gdl{\Phi(\gdl{\varphi})}) \to \Phi(\gdl{\neg \varphi})$ by $\BD{2}$. 
By $\D{3}$, $S \vdash \Phi(\gdl{\varphi}) \to \Phi(\gdl{\Phi(\gdl{\varphi})})$. 
Thus $S \vdash \Phi(\gdl{\varphi}) \to \Phi(\gdl{\neg \varphi})$, and hence $S \vdash \neg \varphi \to \exists x(\Fml(x) \land \Phi(x) \land \Phi(\dot{\neg} x))$. 
It follows $S \vdash \Con^H \to \varphi$. 
By Proposition \ref{LP3}, $T \nvdash \varphi$, and thus $T \nvdash \Con^H$. 
\end{proof}

Jeroslow \cite{Jer73} proved that if $\mathcal{L}_A$ contains sufficiently many primitive recursive terms and if $\Phi(x)$ satisfies $\D{1}$ and $S \vdash \Phi(t) \to \Phi(\gdl{\Phi(t)})$ for all primitive recursive terms $t$, then $T \nvdash \Con^H$. 
That is to say, in Theorem \ref{G2-2}, if we strengthen the condition $\D{3}$ in this way, then the condition $\BD{2}$ can be omitted.  
As a consequence, Jeroslow remarked that if $\Phi(x)$ is a $\Gamma$ formula, then the conditions $\D{1}$ and $\GC$ are sufficient for the unprovability of $\Con^H$ in Jersolow's setting of language. 
We show that this is also the case without using such sufficiently many primitive recursive terms. 

\begin{thm}[Jeroslow \cite{Jer73}; Kreisel and Takeuti \cite{KT74}]\label{Jer}
If $\Phi(x)$ is a $\Gamma$ formula satisfying $\D{1}$ and $\GC$, then $T \nvdash \Con^H$. 
\end{thm}
\begin{proof}
Let $\varphi$ be a $\Gamma$ sentence such that $\PA \vdash \varphi \leftrightarrow \Phi(\gdl{\neg \varphi})$. 
By Proposition \ref{LP3}, $T \nvdash \neg \varphi$ because of $\D{1}$. 
By $\GC$ and the choice of $\varphi$, $S \vdash \varphi \to \Phi(\gdl{\varphi}) \land \Phi(\gdl{\neg \varphi})$. 
Then we have $S \vdash \varphi \to \neg \Con^H$. 
Therefore $T \nvdash \Con^H$. 
\end{proof}

By Proposition \ref{LP1}.8 and Theorem \ref{G2}, if $\Phi(x)$ is a $\Sigma_1$ formula satisfying $\D{1}$, $\D{2}$ and $\PC$, then $T \nvdash \Con^L$. 
Also by Proposition \ref{LP1}.6 and Theorem \ref{Jer}, if $\Phi(x)$ is a $\Sigma_1$ formula satisfying $\D{1}$, $\BD{2}$ and $\PC$, then $T \nvdash \Con^H$. 
We improve the latter statement as follows. 

\begin{thm}\label{G2-3}
If $\Phi(x)$ is a $\Sigma_1$ formula satisfying $\D{1}$ and $\PC$, then $T \nvdash \Con^H$.
\end{thm}
\begin{proof}
Suppose that $\Phi(x)$ is $\Sigma_1$ and satisfies $\D{1}$ and $\PC$. 
Let $T_0$ be a finite subtheory of $T$ containing $\mathsf{Q}$. 
Let $\varphi$ be a $\Sigma_1$ sentence satisfying $\PA \vdash \varphi \leftrightarrow \Phi(\gdl{\neg(\bigwedge T_0 \to \varphi)})$. 
By $\PC$ and formalized deduction theorem, we have $S \vdash \PR_{[T_0]}(\gdl{\varphi}) \to \Phi(\gdl{\bigwedge T_0 \to \varphi})$. 
By $\SC$ for $\PR_{[T_0]}(x)$, $S \vdash \varphi \to\Phi(\gdl{\bigwedge T_0 \to \varphi})$. 
Since $\PA \vdash \varphi \to \Phi(\gdl{\neg(\bigwedge T_0 \to \varphi)})$ by the choice of $\varphi$, we obtain $S \vdash \varphi \to \neg \Con^H$. 

If $T \vdash \Con^H$, then $T \vdash \neg \varphi$. 
Also $T \vdash \bigwedge T_0 \land \neg \varphi$, and this means $T \vdash \neg(\bigwedge T_0 \to \varphi)$. 
By $\D{1}$, $S \vdash \Phi(\gdl{\neg (\bigwedge T_0 \to \varphi)})$, and hence $S \vdash \varphi$. 
This is a contradiction. 
Therefore $T \nvdash \Con^H$. 
\end{proof}

\begin{rem}
The following makeshift condition $\SC^-$ is of course weaker than $\SC$ if $\bigwedge \emptyset \to \varphi$ is identical to $\varphi$.  
\begin{description}
	\item [$\SC^-$] There exists a finite subtheory $T_0$ of $T$ such that for any $\Sigma_1$ sentence $\varphi$, $S \vdash \varphi \to \Phi(\gdl{\bigwedge T_0 \to \varphi})$. 
\end{description}
Our proof of Proposition \ref{LP1}.6 $(\Rightarrow)$ actually shows two implications ``$\PC \Rightarrow \SC^-$'' and ``$\{\BD{2}, \SC^-\} \Rightarrow \SC$''. 
Also our proof of Theorem \ref{G2-3} essentially shows that if $\Phi(x)$ is a $\Sigma_1$ formula satisfying $\D{1}$ and $\SC^-$, then $T \nvdash \Con^H$. 
Then Theorem \ref{Jer} in the case $\Gamma = \Sigma_1$ and Theorem \ref{G2-3} directly follow from these observations. 
\end{rem}

In this section, we have seen that $\{\D{1}, \D{2}, \D{3}\}$ is sufficient for $T \nvdash \Con^L$ (Theorem \ref{G2}), and $\{\D{1}, \BD{2}, \D{3}\}$ is sufficient for $T \nvdash \Con^H$ (Theorem \ref{G2-2}). 
Also for $\Sigma_1$ formulas $\Phi(x)$, each of $\{\D{1}, \SC\}$ and $\{\D{1}, \PC\}$ is sufficient for $T \nvdash \Con^H$ (Theorems \ref{Jer} and \ref{G2-3}). 
From examples of formulas given in Section \ref{Sec:W}, the following non-implications are obtained. 
These non-implications show that these unprovability results are optimal. 
For example, the third clause in the following list means that there exists a $\Sigma_1$ formula $\Phi(x)$ satisfying both $\D{1}$ and $\D{2}$ such that $T \vdash \Con^H$. 

\begin{itemize}
	\item $\{\D{1}, \D{2}, \SC\} \not \Rightarrow T \nvdash \Con^H$ (Fact \ref{exFef}). 
	\item $\{\Phi \in \Sigma_1, \D{2}, \D{3}, \SC, \PC\} \not \Rightarrow T \nvdash \Con^H$ (Proposition \ref{exQ}). 
	\item $\{\Phi \in \Sigma_1, \D{1}, \D{2}\} \not \Rightarrow T \nvdash \Con^H$ (Fact \ref{exAra}.1). 
	\item $\{\Phi \in \Sigma_1, \D{1}, \D{3}\} \not \Rightarrow T \nvdash \Con^H$ (Fact \ref{exAra}.2). 
	\item $\{\Phi \in \Sigma_1, \D{1}, \BD{2}, \D{3}\} \not \Rightarrow T \nvdash \Con^L$ (Fact \ref{exKur}.3). 
	\item $\{\Phi \in \Sigma_1, \D{1}, \SC, \PC\} \not \Rightarrow T \nvdash \Con^L$ (Proposition \ref{exMos}). 
	\item $\{\Phi \in \Sigma_1, \D{1}, \D{2}, \SC\} \not \Rightarrow T \nvdash \Con^{\Sigma_1}$ (Proposition \ref{WP2}). 
\end{itemize}

These non-implications show that none of $\{\D{1}, \BD{2}, \D{3}\}$, $\{\D{1}, \SC\}$ and $\{\D{1}, \PC\}$ implies $\{\D{1}, \D{2}, \D{3}\}$. 
Moreover we obtain the following non-implications. 
\begin{itemize}
	\item $\{\Phi \in \Sigma_1, \D{1}, \D{2}, \D{3}\} \not \Rightarrow \SC$ (Proposition \ref{WP4}). 
	By Proposition \ref{LP1}.6, this is equivalent to $\{\Phi \in \Sigma_1, \D{1}, \D{2}, \D{3}\} \not \Rightarrow \PC$. 
	\item $\{\Phi \in \Sigma_1, \D{1}, \SC, \PC\} \not \Rightarrow \BD{2}$ (Proposition \ref{exMos}). 
	\item $\{\Phi \in \Sigma_1, \D{1}, \SC\} \not \Rightarrow \PC$ (Proposition \ref{WP5}). 
	\item $\{\Phi \in \Sigma_1, \D{1}, \PC\} \not \Rightarrow \SC$ (Proposition \ref{WP6}). 
\end{itemize}
Consequently, $\{\D{1}, \BD{2}, \D{3}\}$, $\{\D{1}, \SC\}$ and $\{\D{1}, \PC\}$ are pairwise incomparable. 
Also $\{\D{1}, \D{2}, \D{3}\}$ is incomparable with each of $\{\D{1}, \SC\}$ and $\{\D{1}, \PC\}$.

\subsection{Uniform derivability conditions}

In this subsection, we introduce and investigate uniform derivability conditions. 
Let $\varphi(\vec{x})$ be an abbreviation for $\varphi(x_0, \ldots, x_{k})$ for some $k$. 

\begin{defn}[Uniform derivability conditions]\leavevmode
\begin{description}
	\item [$\DU{1}$] If $T \vdash \forall \vec{x}\,\varphi(\vec{x})$, then $S \vdash \forall \vec{x}\, \Phi(\gdl{\varphi(\vec{\dot{x}})})$ for any formula $\varphi(\vec{x})$. 

	\item [$\DU{2}$] $S \vdash \forall \vec{x}\, (\Phi(\gdl{\varphi(\vec{\dot{x}}) \to \psi(\vec{\dot{x}})}) \to (\Phi(\gdl{\varphi(\vec{\dot{x}})}) \to \Phi(\gdl{\psi(\vec{\dot{x}})})))$ for any formulas $\varphi(\vec{x})$ and $\psi(\vec{x})$. 

	\item [$\DU{3}$] $S \vdash \forall \vec{x}\, (\Phi(\gdl{\varphi(\vec{\dot{x}})}) \to \Phi(\gdl{ \Phi(\gdl{\varphi(\vec{\dot{x}})})}))$ for any formula $\varphi(\vec{x})$. 

	\item [$\GCU$] $S \vdash \forall \vec{x}\, (\varphi(\vec{x}) \to \Phi(\gdl{\varphi(\vec{\dot{x}})}))$ for any $\Gamma$ formula $\varphi(\vec{x})$. 

	\item [$\BDU{m}$ ($m \geq 1$)] If $\displaystyle T \vdash \forall \vec{x}\left(\bigwedge_{0 < i < m} \varphi_i(\vec{x}) \to \varphi_m(\vec{x}) \right)$, \\
	then $\displaystyle S \vdash \forall \vec{x} \left( \bigwedge_{0 < i < m}\Phi(\gdl{\varphi_i(\vec{\dot{x}})}) \to \Phi(\gdl{\varphi_m(\vec{\dot{x}})}) \right)$\\
for any formulas $\varphi_1(\vec{x}), \ldots, \varphi_m(\vec{x})$. 

	\item [$\CB$] $S \vdash \Phi(\gdl{\forall \vec{x}\, \varphi(\vec{x})}) \to \forall \vec{x}\, \Phi(\gdl{\varphi(\vec{\dot{x}})})$ for any formula $\varphi(\vec{x})$. 

	\item [$\PCU$] $S \vdash \forall \vec{x} (\PRL(\gdl{\varphi(\vec{\dot{x}})}) \to \Phi(\gdl{\varphi(\vec{\dot{x}})}))$ for any formula $\varphi(\vec{x})$. 
\end{description}
\end{defn}

Usual proofs of the Hilbert--Bernays--L\"ob derivability conditions $\D{1}$, $\D{2}$ and $\D{3}$ (in books such as \cite{Boo93}) are demonstrated by showing stronger uniform derivability conditions $\DU{1}$, $\DU{2}$ and $\SCU$. 
Notice that the natural provability predicates $\PR_T(x)$ satisfy full uniform derivability conditions. 

As in the local version, the conditions $\BDU{m}$ ($m \geq 1$) were introduced by Buchholz \cite{Buc93}, and $\BDU{1}$ is precisely $\DU{1}$. 
The condition $\CB$ claims that sentences corresponding to the Converse Barcan Formula investigated in predicate modal logic (see \cite{HC}) are provable. 
Notice that the condition $\HB{2}$ described in the introduction seems to be a variant of the condition $\CB$. 
It is easy to see that each of uniform derivability conditions is stronger than the corresponding local version. 
Moreover, uniform derivability conditions are strictly stronger than local derivability conditions (see Proposition \ref{WP1} in Section \ref{Sec:W}). 

As in the local version, we obtain the following proposition. 

\begin{prop}\label{UP1}\leavevmode
\begin{enumerate}
	\item $\DC$ and $\BDU{m}$ for some $m \geq 1 \Rightarrow \DU{1}$. 

	\item $\BDU{3} \Rightarrow \DU{2}$. 

	\item The following are equivalent: 
	\begin{enumerate}
		\item $\DU{1}$ and $\DU{2}$. 
		\item $\BDU{m}$ for all $m \geq 1$. 
		\item $\DU{1}$ and $\BDU{m}$ for some $m \geq 3$. 
	\end{enumerate}

	\item If $\Phi(x)$ is a $\Gamma$ formula, then $\GCU \Rightarrow \DU{3}$. 

	\item $\BDU{2}$ and $\PCU \iff \BDU{2}$ and $\SCU$. 
	
	\item $\BDU{2}$ and $\PCU \Rightarrow \DU{1}$.  

	\item $\DU{1}$, $\DU{2}$ and $\PCU \iff \DU{1}$, $\DU{2}$ and $\SCU$. 

\end{enumerate}
\end{prop}

The condition $\CB$ is related to other conditions. 

\begin{prop}\label{UP2}\leavevmode
\begin{enumerate}
	\item $\D{1}$ and $\CB \Rightarrow \DU{1}$.
	\item $\BDU{2} \Rightarrow \CB$. 
	\item $\DU{2}$ and $\PCU \Rightarrow \CB$. 
	\item The following are equivalent: 
	\begin{enumerate}
		\item $\DU{1}$ and $\DU{2}$. 
		\item $\D{1}$, $\BDU{2}$ and $\DU{2}$. 
		\item $\D{1}$, $\CB$ and $\DU{2}$. 
	\end{enumerate}
\end{enumerate}
\end{prop}
\begin{proof}
1. Suppose that $\Phi(x)$ satisfies $\D{1}$ and $\CB$. 
Assume $T \vdash \forall \vec{x}\, \varphi(\vec{x})$. 
Then $S \vdash \Phi(\gdl{\forall \vec{x}\, \varphi(\vec{x})})$ by $\D{1}$. 
Since $S \vdash \Phi(\gdl{\forall \vec{x}\, \varphi(\vec{x})}) \to \forall \vec{x}\, \Phi(\gdl{\varphi(\vec{\dot{x}})})$ by $\CB$, we have $S \vdash \forall \vec{x}\, \Phi(\gdl{\varphi(\vec{\dot{x}})})$. 

2. Suppose that $\Phi(x)$ satisfies $\BDU{2}$. 
Since $T \vdash \forall \vec{x}\, \varphi(\vec{x}) \to \varphi(\vec{x})$, we have $S \vdash \Phi(\gdl{\forall \vec{x}\, \varphi(\vec{x})}) \to \Phi(\gdl{\varphi(\vec{\dot{x}})})$ by $\BDU{2}$. 
Therefore $S \vdash \Phi(\gdl{\forall \vec{x}\, \varphi(\vec{x})}) \to \forall \vec{x}\, \Phi(\gdl{\varphi(\vec{\dot{x}})})$. 

3. Suppose $\Phi(x)$ satisfies $\DU{2}$ and $\PCU$. 
Let $\varphi(\vec{x})$ be any formula. 
Since $\forall \vec{x} \varphi(\vec{x}) \to \varphi(\vec{x})$ is provable in predicate calculus, $S \vdash \PRL(\gdl{\forall \vec{x} \varphi(\vec{x}) \to \varphi(\vec{\dot{x}})})$ by $\DU{1}$ for $\PRL(x)$. 
From $\PCU$, $S \vdash \Phi(\gdl{\forall \vec{x} \varphi(\vec{x}) \to \varphi(\vec{\dot{x}})})$. 
Then by $\DU{2}$, $S \vdash \Phi(\gdl{\forall \vec{x} \varphi(\vec{x})}) \to \Phi(\gdl{\varphi(\vec{\dot{x}})})$. 
Thus $S \vdash \Phi(\gdl{\forall \vec{x}\, \varphi(\vec{x})}) \to \forall \vec{x}\, \Phi(\gdl{\varphi(\vec{\dot{x}})})$. 

4. The implications $(a) \Rightarrow (b)$, $(b) \Rightarrow (c)$ and $(c) \Rightarrow (a)$ follow from Proposition \ref{UP1}.3, clause 2 and clause 1, respectively. 
\end{proof}

The following corollary immediately follows from clauses 1, 2 and 3 of Proposition \ref{UP2}. 

\begin{cor}\label{UC1}\leavevmode
\begin{enumerate}
	\item $\D{1}$ and $\BDU{2} \Rightarrow \DU{1}$. 
	\item $\D{1}$, $\DU{2}$ and $\PCU \Rightarrow \DU{1}$. 
\end{enumerate}
\end{cor}

Hilbert and Bernays \cite{HB39} proved that if a $\Sigma_1$ formula $\Phi(x)$ satisfies the conditions $\HB{1}$, $\HB{2}$ and $\HB{3}$ described in the introduction, then $T \nvdash \Con^H$. 
In our framework, the Hilbert--Bernays derivability conditions can be replaced by the conditions $\BD{2}$, $\CB$ and $\DCU$ without any substantial change. 
Then we obtain the following version of the second incompleteness theorem. 

\begin{thm}[Hilbert and Bernays \cite{HB39}]\label{UHB}
If $\Phi(x)$ is a $\Sigma_1$ formula satisfying $\BD{2}$, $\CB$ and $\DCU$, then $T \nvdash \Con^H$. 
\end{thm}
\begin{proof}
Suppose that $\Phi(x)$ is $\Sigma_1$ and satisfies $\BD{2}$, $\CB$ and $\DCU$. 
Let $\varphi$ be a $\Pi_1$ sentence satisfying $\PA \vdash \varphi \leftrightarrow \neg \Phi(\gdl{\varphi})$. 
Let $\delta(x)$ be a $\Delta_0$ formula with $\PA \vdash \varphi \leftrightarrow \forall x \delta(x)$. 
Then by $\BD{2}$, $S \vdash \Phi(\gdl{\varphi}) \to \Phi(\gdl{\forall x \delta(x)})$. 
By $\CB$, we obtain
\begin{align}\label{eq2}
	S \vdash \neg \varphi \to \forall x \Phi(\gdl{\delta(\dot{x})}). 
\end{align}

On the other hand, $S \vdash \neg \delta(x) \to \Phi(\gdl{\neg \delta(\dot{x})})$ by $\DCU$.  
Then $S \vdash \exists x \neg \delta(x) \to \exists x \Phi(\gdl{\neg \delta(\dot{x})})$. 
Hence $S \vdash \neg \varphi \to \exists x \Phi(\gdl{\neg \delta(\dot{x})})$. 
By combining this with (\ref{eq2}), we obtain 
\[
	S \vdash \neg \varphi \to \exists x (\Phi(\gdl{\delta(\dot{x})}) \land \Phi(\gdl{\neg \delta(\dot{x})})).
\]
It follows $S \vdash \neg \varphi \to \exists x (\Fml(x) \land \Phi(x) \land \Phi(\dot{\neg} x))$, and hence $S \vdash \Con^H \to \varphi$. 
By Proposition \ref{LP1}.2, $\Phi(x)$ satisfies $\D{1}$. 
Then by Proposition \ref{LP3}, $T \nvdash \varphi$. 
Therefore we conclude $T \nvdash \Con^H$. 
\end{proof}

Theorem \ref{UHB} is optimal in the sense of the following non-implications from Section \ref{Sec:W}. 
\begin{itemize}
	\item $\{\D{1}, \BD{2}, \CB, \DCU\} \not \Rightarrow T \nvdash \Con^H$ (Fact \ref{exFef}). 
	\item $\{\Phi \in \Sigma_1, \CB, \DCU\} \not \Rightarrow T \nvdash \Con^H$ (Proposition \ref{exQ}). 
	\item $\{\Phi \in \Sigma_1, \BD{2}, \CB\} \not \Rightarrow T \nvdash \Con^H$ (Proposition \ref{exN}). 
	\item $\{\Phi \in \Sigma_1, \D{1}, \BD{2}, \DCU\} \not \Rightarrow T \nvdash \Con^H$ (Fact \ref{exKur}.1). 
	\item $\{\Phi \in \Sigma_1, \D{1}, \BD{2}, \CB, \DCU\} \not \Rightarrow T \nvdash \Con^L$ (Fact \ref{exKur}.2). 
\end{itemize}
Notice that $\{\BD{2}, \CB, \DCU\}$ is equivalent to $\{\D{1}, \BD{2}, \CB, \DCU\}$ by Proposition \ref{LP1}.2. 
For the latter condition, we do not know if $\{\Phi \in \Sigma_1, \D{1}, \BD{2}, \CB, \DCU\}$ is optimal to conclude $T \nvdash \Con^H$ or not. 

\begin{prob}\leavevmode
\begin{enumerate}
	\item Is there a $\Sigma_1$ provability predicate satisfying $\D{1}$, $\CB$ and $\DCU$ such that $T \vdash \Con^H$?
	\item Is there a $\Sigma_1$ provability predicate satisfying $\D{1}$, $\BD{2}$ and $\CB$ such that $T \vdash \Con^H$?
\end{enumerate}
\end{prob}

The following two non-implications from Section \ref{Sec:W} indicate that $\{\BD{2}, \CB, \DCU\}$ is incomparable with each of $\{\D{1}, \D{2}, \D{3}\}$, $\{\D{1}, \BD{2}, \D{3}\}$, $\{\D{1}, \SC\}$ and $\{\D{1}, \PC\}$.

\begin{itemize}
	\item $\{\Phi \in \Sigma_1, \BD{2}, \CB, \DCU\} \not \Rightarrow \D{3}$ (Fact \ref{exKur}.2). 
	\item $\{\Phi \in \Sigma_1, \D{1}, \D{2}, \SC\} \not \Rightarrow \CB$ (Proposition \ref{WP1}). 
\end{itemize}

Usual proof of $\SCU$ (in books such as \cite{Boo93}) proceeds by induction on the construction of $\Sigma_1$ formulas, and it requires much effort. 
In the lecture note \cite{Buc93} by Buchholz, an elegant schematic proof of $\SCU$ is presented. 
More precisely, it is proved that for a proof of $\SCU$, the assumption ``$\BDU{m}$ for all $m \geq 1$'' is sufficient. 
By Proposition \ref{UP1}.3, this assumption is equivalent to $\{\DU{1}, \DU{2}\}$. 
Hence Buchholz's work is stated as follows. 

\begin{thm}[Buchholz \cite{Buc93}]\label{UBuc}
$\DU{1}$ and $\DU{2} \Rightarrow \SCU$. 
\end{thm}

In Rautenberg's book \cite{Rau10}, a schematic proof of $\SCU$ based on Buchholz's argument is presented. 
As a corollary to Theorem \ref{UBuc}, we obtain the following version of the second incompleteness theorem. 

\begin{cor}\label{UC2}
If $\Phi(x)$ is a $\Sigma_1$ formula satisfying $\DU{1}$ and $\DU{2}$, then $T \nvdash \Con^L$. 
\end{cor}

Notice that \{$\DU{1}, \DU{2}\}$ implies $\{\D{1}, \BDU{2}\}$ by Proposition \ref{UP1}.3. 
The following theorem improves Buchholz's Theorem \ref{UBuc} which will be proved in the next section. 

\begin{thm}\label{MT}
$\D{1}$ and $\BDU{2} \Rightarrow \SCU$. 
\end{thm}

This theorem says that only the $m = 1, 2$ cases of Buchholz's assumption are sufficient to prove $\SCU$. 
We will also prove that Theorem \ref{MT} is actually an improvement of Theorem \ref{UBuc} (see Theorem \ref{MT2} below). 
Interestingly, for $\Sigma_1$ formulas, $\{\D{1}, \BDU{2}\}$ implies $\{\D{1}, \BD{2}, \D{3}\}$, $\{\D{1}, \SC\}$, $\{\D{1}, \PC\}$ and $\{\BD{2}, \CB, \DCU\}$ by Theorem \ref{MT} and Proposition \ref{UP1}, and each of them is sufficient for $T \nvdash \Con^H$. 
As a consequence, we obtain the following corollary. 

\begin{cor}\label{UC3}
If $\Phi(x)$ is a $\Sigma_1$ formula satisfying $\D{1}$ and $\BDU{2}$, then $T \nvdash \Con^H$. 
\end{cor}

Related to Corollary \ref{UC3}, we propose the following problem. 

\begin{prob}
Is there a $\Sigma_1$ formula $\Phi(x)$ satisfying $\D{1}$ and $\BDU{2}$ such that $T \vdash \Con^L$? 
\end{prob}

In contrast to the consistency statements $\Con^H$ and $\Con^L$, Proposition \ref{WP2} in Section \ref{Sec:W} shows that the full uniform derivability conditions are not sufficient for the unprovability of $\Con^{\Sigma_1}$ and $\Con^G$. 

From Theorem \ref{MT} and Proposition \ref{UP1}.5, we obtain the following corollary. 

\begin{cor}
$\D{1}$ and $\BDU{2} \Rightarrow \PCU$. 
\end{cor}


Moreover, we show that $\D{1}$ and $\BDU{2}$ imply a stronger version of $\PCU$. 
For $n \geq 0$, let $\mathsf{True}_{\Sigma_n}(x)$ be a natural formula saying that ``$x$ is a true $\Sigma_n$ sentence'' (cf.~H\'ajek and Pudl\'ak \cite{HP93}). 

\begin{prop}\label{UP1.5}
If $\Phi(x)$ satisfies $\D{1}$ and $\BDU{2}$, then for $n \geq 0$, 
\[
	S \vdash \forall x(\Sigma_n(x) \land \PRL(x) \to \Phi(\gdl{\mathsf{True}_{\Sigma_n}(\dot{x})})). 
\]
\end{prop}
\begin{proof}
Suppose that $\Phi(x)$ satisfies $\D{1}$ and $\BDU{2}$, and let $n \geq 0$. 
By Theorem \ref{MT}, $\Phi(x)$ satisfies $\SCU$, and hence $S \vdash \Sigma_n(x) \land \PRL(x) \to \Phi(\gdl{\Sigma_n(\dot{x}) \land \PRL(\dot{x})})$. 
By reflexiveness, $T \vdash \Sigma_n(x) \land \PRL(x) \to \mathsf{True}_{\Sigma_n}(x)$. 
Then $S \vdash \Phi(\gdl{\Sigma_n(\dot{x}) \land \PRL(\dot{x})}) \to \Phi(\gdl{\mathsf{True}_{\Sigma_n}(\dot{x})})$ by $\BDU{2}$. 
We conclude $S \vdash \Sigma_n(x) \land \PRL(x) \to \Phi(\gdl{\mathsf{True}_{\Sigma_n}(\dot{x})})$. 
\end{proof}

\subsection{Global derivability conditions}

At last, we introduce the strongest version of derivability conditions. 
They are called global derivability conditions. 

\begin{defn}[Global derivability conditions]\leavevmode
\begin{description} 
	\item [$\DG{2}$] $S \vdash \forall x \forall y(\Fml(x) \land \Fml(y) \to (\Phi(x \dot{\to} y) \to (\Phi(x) \to \Phi(y))))$. 
	\item [$\DG{3}$] $S \vdash \forall x (\Fml(x) \to (\Phi(x) \to \Phi(\gdl{\Phi(\dot{x})})))$. 
	\item [$\GCG$] $S \vdash \forall x(\mathsf{True}_\Gamma(x) \to \Phi(x))$. 
	\item [$\PCG$] $S \vdash \forall x(\Fml(x) \to (\PRL(x) \to \Phi(x)))$. 
\end{description}
\end{defn}
The condition $\DG{2}$ for provability predicates $\PR_T(x)$ was proved in Feferman \cite{Fef60}. 
Montagna \cite{Mon79} investigated the condition $\DG{2}$. 
The condition $\SCG$ for $\PR_\mathsf{Q}(x)$ is explicitly stated in the book \cite{HP93}. 
Global derivability conditions are strictly stronger than uniform derivability conditions (see Proposition \ref{WP2}).

We can prove the following proposition as in the uniform version. 

\begin{prop}\label{GP1}\leavevmode
\begin{enumerate}
	\item If $\Phi(x)$ is a $\Gamma$ formula, then $\GCU \Rightarrow \DG{3}$. 
	\item $\D{1}$, $\DG{2}$ and $\PCG \Rightarrow \SCG$. 
\end{enumerate}
\end{prop}

Proposition \ref{GP1}.2 was stated in von B\"ulow \cite{vB08} and Visser \cite{Vis}. 

Consistency statements are enhanced by global derivability conditions. 

\begin{prop}\label{GP2}\leavevmode
\begin{enumerate}
	\item If $\Phi(x)$ satisfies $\DG{2}$ and $\PCG$, then $S \vdash \Con^G \to \Con^H$. 
	\item If $\Phi(x)$ satisfies $\D{1}$, $\DG{2}$ and $\PCG$, then $\Con^H$, $\Con^L$ and $\Con^G$ are mutually equivalent in $S$. 
	\item If $\Phi(x)$ satisfies $\DG{2}$ and $\SCG$, then $\Con^L$ and $\Con^{\Sigma_1}$ are equivalent in $S$. 
\end{enumerate}
\end{prop}
\begin{proof}
1. Suppose $\Phi(x)$ satisfies $\DG{2}$ and $\PCG$. 
Since $\PA \vdash \forall x \forall y (\Fml(x) \land \Fml(y) \to \PRL(x \dot{\to} (\dot{\neg} x \dot{\to} y)))$, $S \vdash \forall x \forall y (\Fml(x) \land \Fml(y) \to \Phi(x \dot{\to} (\dot{\neg} x \dot{\to} y)))$ by $\PCG$. 
Hence $\forall x \forall y (\Fml(x) \land \Fml(y) \land \Phi(x) \land \Phi(\dot{\neg} x) \to \Phi(y))$ is provable in $S$ by $\DG{2}$. 
This sentence is equivalent to $\Con^G \to \Con^H$. 

2. This follows from Proposition \ref{LP2} and clause 1. 

3. Suppose $\Phi(x)$ satisfies $\DG{2}$ and $\SCG$. 
By Proposition \ref{LP2}, it suffices to show $S \vdash \Con^{\Sigma_1} \to \Con^L$. Since $\PA \vdash \neg \mathsf{True}_{\Sigma_1}(\gdl{0 \neq 0})$, $\PA \vdash \Sigma_1(x) \land \Sent(x) \to \mathsf{True}_{\Sigma_1}(\gdl{0 \neq 0} \dot{\to} x)$. 
By $\SCG$, $S \vdash \Sigma_1(x) \land \Sent(x) \to \Phi(\gdl{0 \neq 0} \dot{\to} x)$. 
By $\DG{2}$, $S \vdash \Sigma_1(x) \land \Sent(x) \to (\Phi(\gdl{0 \neq 0}) \to \Phi(x))$. 
Thus $S \vdash \Con^{\Sigma_1} \to \Con^L$. 
\end{proof}

From Theorems \ref{G2} and \ref{G2-3}, and Proposition \ref{GP2}, we obtain the following corollary. 

\begin{cor}\label{GC1}\leavevmode
\begin{enumerate}
	\item If $\Phi(x)$ is a $\Sigma_1$ formula satisfying $\D{1}$, $\DG{2}$ and $\PCG$, then $T \nvdash \Con^G$. 
	\item If $\Phi(x)$ is a $\Sigma_1$ formula satisfying $\D{1}$, $\DG{2}$ and $\SCG$, then $T \nvdash \Con^{\Sigma_1}$. 
\end{enumerate}
\end{cor}

Corollary \ref{UC1}.2 and Proposition \ref{GP1}.2 show that $\{\DU{1}, \DG{2}, \SCG\}$ is weaker than $\{\D{1}, \DG{2}, \PCG\}$. 
Moreover, Proposition \ref{WP3} in Section \ref{Sec:W} shows the following interesting non-implication:  
\begin{itemize}
	\item $\{\Phi \in \Sigma_1, \DU{1}, \DG{2}, \SCG\} \not \Rightarrow T \nvdash \Con^G$. 
\end{itemize}
Hence in contrast to local and uniform versions, $\{\DU{1}, \DG{2}, \SCG\}$ is strictly weaker than $\{\D{1}, \DG{2}, \PCG\}$. 
Also this non-implication indicates that global derivability conditions except for $\PCG$ are not sufficient for the unprovability of G\"odel's consistency statement $\Con^G$ even if $\Phi$ is $\Sigma_1$. 
This shows that neither Hilbert--Bernays' conditions nor L\"ob's conditions accomplish G\"odel's original statement of the second incompleteness theorem.

Let $\mathsf{LogAx}(x)$ be a suitable $\Delta_1$ formula representing the set of all logical axioms of predicate calculus formulated in Feferman's paper \cite{Fef60}. 
In Feferman's formulation, the sole inference rule is modus ponens, and the generalization rule is admissible (see Result 2.1 in \cite{Fef60}). 
The following condition was introduced by Montagna \cite{Mon79}. 

\begin{defn}
\leavevmode
\begin{description}
	\item [$\Ax$] $S \vdash \forall x(\mathsf{LogAx}(x) \to \Phi(x))$. 
\end{description}
\end{defn}

The condition $\Ax$ is related to the condition $\PCG$. 

\begin{prop}\label{GP3}\leavevmode
\begin{enumerate}
	\item $\PCG \Rightarrow \Ax$. 
	\item $\DG{2}$ and $\Ax \Rightarrow \PCG$. 
	\item If $\Phi(x)$ satisfies $\D{1}$, then for any sentence $\varphi$, $S \vdash \mathsf{LogAx}(\gdl{\varphi}) \to \Phi(\gdl{\varphi})$. 
\end{enumerate}
\end{prop}
\begin{proof}
1. This is because $\PA \vdash \forall x(\mathsf{LogAx}(x) \to \PRL(x))$. 

2. Let $\PRL'(x)$ be a natural provability predicate of the predicate calculus formulated in Feferman's framework. 
Then $\PA \vdash \forall x(\Fml(x) \to (\PRL(x) \to \PRL'(x)))$ holds by induction inside $\PA$. 
Since $S$ proves that $\Phi(x)$ contains axioms of $\PRL'(x)$ by $\Ax$ and that $\Phi(x)$ is closed under the inference rule of $\PRL'(x)$ by $\DG{2}$, $S$ proves $\forall x (\Fml(x) \to (\PRL'(x) \to \Phi(x)))$ by induction inside $S$. 
Hence $S \vdash \forall x(\Fml(x) \to (\PRL(x) \to \Phi(x)))$ holds. 

3. Let $\varphi$ be any sentence. 
If $\varphi$ is a logical axiom, then $T \vdash \varphi$. 
By $\D{1}$, $S \vdash \Phi(\gdl{\varphi})$. 
If $\varphi$ is not a logical axiom, then $S \vdash \neg \mathsf{LogAx}(\gdl{\varphi})$. 
In either case, we obtain $S \vdash \mathsf{LogAx}(\gdl{\varphi}) \to \Phi(\gdl{\varphi})$. 
\end{proof}

Montagna \cite{Mon79} proved that if $\Phi(x)$ satisfies $\D{1}$, $\DG{2}$ and $\Ax$, then $\D{3}$ is redundant for a proof of L\"ob's theorem. 
From Propositions \ref{GP1} and \ref{GP3}, and Corollaries \ref{UC1}.2 and \ref{GC1}, we obtain the following improvement of Montagna's result. 

\begin{cor}[Montagna \cite{Mon79}]\label{Mon2}\leavevmode
\begin{enumerate}
	\item $\D{1}$, $\DG{2}$ and $\Ax \Rightarrow \DU{1}$ and $\SCG$. 
	\item If $\Phi(x)$ is a $\Sigma_1$ formula satisfying $\D{1}$, $\DG{2}$ and $\Ax$, then $T \nvdash \Con^G$. 
\end{enumerate}
\end{cor}

\section{Proof of Theorem \ref{MT}}\label{Sec:MT}

In this section, we prove Theorem \ref{MT}, that is, we prove that if $\Phi(x)$ satisfies $\D{1}$ and $\BDU{2}$, then $\Phi(x)$ satisfies $\SCU$. 
Thus in the rest of this section, we fix a formula $\Phi(x)$ satisfying $\D{1}$ and $\BDU{2}$. 
Then by Corollary \ref{UC1}.1, $\Phi(x)$ also satisfies $\DU{1}$. 
First, we prove a lemma, that is an essential application of the condition $\BDU{2}$. 

\begin{lem}\label{MTL1}
Let $\varphi(\vec{x})$ and $\psi(\vec{x})$ be any formulas. 
If $S \vdash \varphi(\vec{x}) \to \Phi(\gdl{\varphi(\vec{\dot{x}})})$ and $\PA \vdash \varphi(\vec{x}) \leftrightarrow \psi(\vec{x})$, then $S \vdash \psi(\vec{x}) \to \Phi(\gdl{\psi(\vec{\dot{x}})})$. 
\end{lem}
\begin{proof}
If $\PA \vdash \varphi(\vec{x}) \leftrightarrow \psi(\vec{x})$, then by $\BDU{2}$, we have 
\[
	S \vdash \Phi(\gdl{\varphi(\vec{\dot{x}})}) \leftrightarrow \Phi(\gdl{\psi(\vec{\dot{x}})}). 
\]
Then the lemma follows immediately.  
\end{proof}

We may assume that every $\Sigma_1$ $\mathcal{L}_A$-formula is $\PA$-provably equivalent to some $\Sigma_1$ formula written in the language $\{0, \mathsf{s}, +, \times\}$. 
Therefore, in proving Theorem \ref{MT}, it suffices to show $S \vdash \sigma(\vec{x}) \to \Phi(\gdl{\sigma(\vec{\dot{x}})})$ for any $\Sigma_1$ formula $\sigma(\vec{x})$ in the language $\{0, \mathsf{s}, +, \times\}$. 
Hence in the rest of this section, we assume that our terms and formulas are written in $\{0, \mathsf{s}, +, \times\}$. 
Before proving Theorem \ref{MT}, we prepare several lemmas. 

\begin{lem}\label{MTL0}
For any formula $\varphi(\vec{y}, v)$, 
\[
	\PA \vdash \gdl{\varphi(\vec{\dot{y}}, \dot{v})}[\mathsf{s}(x) \slash v] = \gdl{\varphi(\vec{\dot{y}}, \mathsf{s}(\dot{x}))}, 
\]
where $\gdl{\varphi(\vec{\dot{y}}, \dot{v})}[\mathsf{s}(x) \slash v]$ is the result of substituting $\mathsf{s}(x)$ for $v$ of $\gdl{\varphi(\vec{\dot{y}}, \dot{v})}$. 
\end{lem}
\begin{proof}
This is because our numeral $\overline{n}$ is defined by applying $\mathsf{s}$ to $0$ $n$ times. 
Then the lemma can be proved by induction on the constructions of terms and formulas. 
We give only an outline of a proof. 

For example, we assume that our G\"odel number $\gn(t)$ of a term $t$ is defined so that $\gn(\mathsf{s}(t)) = \langle 0, \gn(t) \rangle$, where $\langle \cdot, \cdot \rangle$ is a primitive recursive paring function. 
Then we can define a primitive recursive function $\num(x)$ calculating $n \mapsto \gn(\overline{n})$ satisfying $\num(\mathsf{s}(x)) = \langle 0, \num(x) \rangle$. 
This is proved in $\PA$ and corresponds to $\gdl{\dot{v}}[\mathsf{s}(x) \slash v] = \gdl{\mathsf{s}(\dot{x})}$. 
Then by using properties of $\gdl{\cdot}$ such as $\PA \vdash \gdl{\mathsf{s}(t)} = \langle 0, \gdl{t} \rangle$, we can show $\PA \vdash \gdl{t(\vec{\dot{y}}, \dot{v})}[\mathsf{s} (x) \slash v] = \gdl{t(\vec{\dot{y}}, \mathsf{s}(\dot{x}))}$ for any term $t(\vec{y}, v)$. 
Then we can prove the lemma by using properties of $\gdl{\cdot}$. 
\end{proof}

\begin{lem}\label{MTL2}
Let $\varphi(\vec{x}, v)$ be any formula. 
If $S \vdash \varphi(\vec{x}, v) \to \Phi(\gdl{\varphi(\vec{\dot{x}}, \dot{v})})$, then $S \vdash \exists v \varphi(\vec{x}, v) \to \Phi(\gdl{\exists v \varphi(\vec{\dot{x}}, v)})$. 
\end{lem}
\begin{proof}
Suppose $S \vdash \varphi(\vec{x}, v) \to \Phi(\gdl{\varphi(\vec{\dot{x}}, \dot{v})})$. 
Since $T \vdash \varphi(\vec{x}, v) \to \exists v \varphi(\vec{x}, v)$, we have $S \vdash \Phi(\gdl{\varphi(\vec{\dot{x}}, \dot{v})}) \to \Phi(\gdl{\exists v \varphi(\vec{\dot{x}}, v)})$ by $\BDU{2}$. 
Hence $S \vdash \varphi(\vec{x}, v) \to \Phi(\gdl{\exists v \varphi(\vec{\dot{x}}, v)})$. 
Therefore we conclude $S \vdash \exists v \varphi(\vec{x}, v) \to \Phi(\gdl{\exists v \varphi(\vec{\dot{x}}, v)})$.
\end{proof}

\begin{lem}\label{MTL3}
For any natural number $k$ and any variables $x_0, \ldots, x_k, z_0, \ldots, z_k$, 
\[
	S \vdash \bigwedge_{i \leq k} (z_i = x_i) \to \Phi \left(\gdl{\bigwedge_{i \leq k} (\dot{z}_i = \dot{x}_i)}\right).
\]
\end{lem}
\begin{proof}
Since $T \vdash \bigwedge_{i \leq k} (z_i = z_i)$, we have 
\begin{align}\label{F3}
	S \vdash \Phi \left(\gdl{\bigwedge_{i \leq k} (\dot{z}_i = \dot{z}_i)} \right)
\end{align}
by $\DU{1}$. 
Let $v_0, \ldots, v_k$ be fresh variables. 
By equality axioms of predicate calculus, we have
\[
	\PA \vdash \bigwedge_{i \leq k} (z_i = x_i) \to \left(\Phi \left(\gdl{\bigwedge_{i \leq k} (\dot{v}_i = \dot{z}_i) }\right) \to \Phi \left(\gdl{\bigwedge_{i \leq k} (\dot{v}_i = \dot{x}_i)} \right) \right). 
\]
By substituting $z_i$ for $v_i$, we obtain
\[
	\PA \vdash \bigwedge_{i \leq k} (z_i = x_i) \to \left(\Phi \left(\gdl{\bigwedge_{i \leq k} (\dot{z}_i = \dot{z}_i)} \right) \to \Phi \left(\gdl{\bigwedge_{i \leq k} (\dot{z}_i = \dot{x}_i)} \right) \right). 
\]
By combining this with (\ref{F3}), we now obtain
\[
	S \vdash \bigwedge_{i \leq k} (z_i = x_i) \to \Phi \left(\gdl{\bigwedge_{i \leq k} (\dot{z}_i = \dot{x}_i)} \right). 
\]
\end{proof}

For each term $t(\vec{x})$, let $c(t(\vec{x}))$ be the number of constant and function symbols contained in $t(\vec{x})$. 
We call $c(t(\vec{x}))$ the \textit{complexity} of $t(\vec{x})$.

\begin{lem}\label{MTL4}
For any finite sequence $\{t_i(\vec{x})\}_{i \leq k}$ of terms with $\max_{i \leq k} \{c(t_i(\vec{x}))\} \leq 1$, 
\[
	S \vdash \bigwedge_{i \leq k} (z_i = t_i(\vec{x})) \to \Phi \left(\gdl{\bigwedge_{i \leq k} (\dot{z}_i = t_i(\vec{\dot{x}}))}\right).
\]
\end{lem}
\begin{proof}
We prove by induction on the number $m$ of terms of complexity $1$ in such sequences. 
If a sequence does not contain terms of complexity $1$, then it consists of variables, and hence the lemma holds for the sequence by Lemma \ref{MTL3}. 

Suppose that the lemma holds for such sequences with exactly $m$ terms of complexity $1$. 
Let $\{t_i(\vec{x})\}_{i \leq k}$ be any finite sequence consists of terms of complexity less than or equal to $1$ and having exactly $m+1$ terms of complexity $1$. 
We may assume that $c(t_k) = 1$. 
Let $\xi(\vec{v}) : \equiv \bigwedge_{i < k}(z_i = t_i(\vec{x}))$. 
We distinguish the following four cases. 

Case 1: $t_k(\vec{x})$ is $0$. 
Then by induction hypothesis, 
\[
	S \vdash \xi(\vec{v}) \land z_k = y \to \Phi(\gdl{\xi(\vec{\dot{v}}) \land \dot{z}_k = \dot{y}}).
\]
By substituting $0$ for $y$, we obtain
\[
	S \vdash \xi(\vec{v}) \land z_k = 0 \to \Phi(\gdl{\xi(\vec{\dot{v}}) \land \dot{z}_k = \dot{y}})[0 \slash y].
\]
Since $0$ is a numeral, we have
\[
	S \vdash \xi(\vec{v}) \land z_k = 0 \to \Phi(\gdl{\xi(\vec{\dot{v}}) \land \dot{z}_k = 0}).
\]

Case 2: $t_k(\vec{x})$ is $\mathsf{s}(x)$. 
By induction hypothesis, 
\[
	S \vdash \xi(\vec{v}) \land z_k = y \to \Phi(\gdl{\xi(\vec{\dot{v}}) \land \dot{z}_k = \dot{y}}).
\]
By substituting $\mathsf{s}(x)$ for $y$, we obtain
\[
	S \vdash \xi(\vec{v}) \land z_k = \mathsf{s}(x) \to \Phi(\gdl{\xi(\vec{\dot{v}}) \land \dot{z}_k = \dot{y}})[\mathsf{s}(x) \slash y].
\]
By Lemma \ref{MTL0}, we conclude
\[
	S \vdash \xi(\vec{v}) \land z_k = \mathsf{s}(x) \to \Phi(\gdl{\xi(\vec{\dot{v}}) \land \dot{z}_k = \mathsf{s}(\dot{x})}).
\]

Case 3: $t_k(\vec{x})$ is $x + y$. 
Let $\varphi(y)$ be the formula 
\[
	\forall x(\xi(\vec{v}) \land z_k = x + y \to \Phi(\gdl{\xi(\vec{\dot{v}}) \land \dot{z}_k = \dot{x} + \dot{y}})). 
\]
By induction hypothesis, 
\[
	S \vdash \xi(\vec{v}) \land z_k = x \to \Phi(\gdl{\xi(\vec{\dot{v}}) \land \dot{z}_k = \dot{x}}).
\]
Since $\PA \vdash x = x + 0$, we have $\PA \vdash (\xi(\vec{v}) \land z_k = x) \leftrightarrow (\xi(\vec{v}) \land z_k = x + 0)$. 
Then by Lemma \ref{MTL1}, 
\[
	S \vdash \xi(\vec{v}) \land z_k = x + 0 \to \Phi(\gdl{\xi(\vec{\dot{v}}) \land \dot{z}_k = \dot{x} + 0}). 
\]
This means $S \vdash \varphi(0)$. 

By Lemma \ref{MTL0}, we get
\[
	\PA \vdash \varphi(y) \land \xi(\vec{v}) \land z_k = \mathsf{s}(x) + y \to \Phi(\gdl{\xi(\vec{\dot{v}}) \land \dot{z}_k = \mathsf{s}(\dot{x}) + \dot{y}}). 
\]
Since $\PA \vdash \mathsf{s}(x) + y = x + \mathsf{s}(y)$, we obtain
\[
	S \vdash \varphi(y) \land \xi(\vec{v}) \land z_k = x + \mathsf{s}(y) \to \Phi(\gdl{\xi(\vec{\dot{v}}) \land \dot{z}_k = \dot{x} + \mathsf{s}(\dot{y})}). 
\]
by Lemma \ref{MTL1}. 
Then $S \vdash \varphi(y) \to \varphi(\mathsf{s}(y))$. 
By induction axiom, we conclude $S \vdash \forall y \varphi(y)$. 

Case 4: $t_k(\vec{x})$ is $x \times y$. 
Let $\psi(y)$ be the formula 
\[
	\forall w(\xi(\vec{v}) \land z_k = x \times y + w \to \Phi(\gdl{\xi(\vec{\dot{v}}) \land \dot{z}_k = \dot{x} \times \dot{y} + \dot{w}})). 
\]
By induction hypothesis, 
\[
	S \vdash \xi(\vec{v}) \land z_k = w \to \Phi(\gdl{\xi(\vec{\dot{v}}) \land \dot{z}_k = \dot{w}}).
\]
Since $\PA \vdash w = x \times 0 + w$, we have
\[
	S \vdash \xi(\vec{v}) \land z_k = x \times 0 + w \to \Phi(\gdl{\xi(\vec{\dot{v}}) \land \dot{z}_k = \dot{x} \times 0 + \dot{w}})
\]
by Lemma \ref{MTL1}. 
Therefore $S \vdash \psi(0)$. 

Let $\rho(w)$ be the formula
\[
	\forall u(\xi(\vec{v}) \land z_k = x \times y + (u + w) \to \Phi(\gdl{\xi(\vec{\dot{v}}) \land \dot{z}_k = \dot{x} \times \dot{y} + (\dot{u} + \dot{w})})). 
\]
Then as in Case 3, we can prove $S \vdash \psi(y) \to \rho(0)$ and $S \vdash \rho(w) \to \rho(\mathsf{s}(w))$. 
Hence $S \vdash \psi(y) \to \forall w \rho(w)$. 
Then 
\[
	S \vdash \psi(y) \land \xi(\vec{v}) \land z_k = x \times y + (x + w) \to \Phi(\gdl{\xi(\vec{\dot{v}}) \land \dot{z}_k = \dot{x} \times \dot{y} + (\dot{x} + \dot{w})}). 
\]
Since $\PA \vdash x \times y + (x + w) = x \times \mathsf{s}(y) + w$, we get
\[
	S \vdash \psi(y) \land \xi(\vec{v}) \land z_k = x \times \mathsf{s}(y) + w \to \Phi(\gdl{\xi(\vec{\dot{v}}) \land \dot{z}_k = \dot{x} \times \mathsf{s}(\dot{y}) + \dot{w}}) 
\] 
by Lemma \ref{MTL1}. 
Thus $S \vdash \psi(y) \to \psi(\mathsf{s}(y))$, and hence $S \vdash \forall y \psi(y)$. 
By substituting $0$ for $w$ in $\psi(y)$, we obtain
\[
	S \vdash \xi(\vec{v}) \land z_k = x \times y + 0 \to \Phi(\gdl{\xi(\vec{\dot{v}}) \land \dot{z}_k = \dot{x} \times \dot{y} + 0}). 
\]
Then the required conclusion follows from Lemma \ref{MTL1}. 
\end{proof}

\begin{lem}\label{MTL5}
For any finite sequence $\{t_i(\vec{x})\}_{i \leq k}$ of terms, 
\[
	S \vdash \bigwedge_{i \leq k} (z_i = t_i(\vec{x})) \to \Phi \left(\gdl{\bigwedge_{i \leq k} (\dot{z}_i = t_i(\vec{\dot{x}}))}\right).
\]
\end{lem}
\begin{proof}
We prove by induction on $\max_{i \leq k} \{c(t_i(\vec{x}))\}$. 
If $\max_{i \leq k} \{c(t_i(\vec{x}))\} \leq 1$, then the lemma follows from Lemma \ref{MTL4}. 

Suppose that the lemma holds for every finite sequence $\{t_i(\vec{x})\}_{i \leq k}$ of terms with $\max_{i \leq k} \{c(t_i(\vec{x}))\} = n \geq 1$. 
Then we show that the lemma holds for all finite sequences $\{t_i(\vec{x})\}_{i \leq k}$ containing only terms of complexity less than or equal to $n+1$. 

As in our proof of Lemma \ref{MTL4}, this is proved by induction on the number $m$ of terms of complexity $n+1$ in such sequences. 
If $m = 0$, then the lemma follows from induction hypothesis. 
Then assume that the lemma holds for such sequences with exactly $m$ terms of complexity $n+1$. 

Let $\{t_i\}_{i \leq k}$ be any finite sequence consists of terms of complexity less than or equal to $n+1$ and having exactly $m+1$ terms of complexity $n+1$. 
We may assume that $c(t_k) = n+1$. 
Let $\xi(\vec{v}) : \equiv \bigwedge_{i < k}(z_i = t_i(\vec{x}))$. 
We give only a proof of the case that $t_k(\vec{x})$ is $\mathsf{s}(t'(\vec{x}))$ for some term $t'(\vec{x})$ of complexity $n$. 
Other cases are proved in a similar way. 

Notice that $c(\mathsf{s}(w)) = 1 \leq n$ and $c(t'(\vec{x})) = n$. 
Then by induction hypothesis, 
\[
	S \vdash \xi(\vec{v}) \land z_k = \mathsf{s}(w) \land w = t'(\vec{x}) \to \Phi(\gdl{\xi(\vec{\dot{v}}) \land \dot{z}_k = \mathsf{s}(\dot{w}) \land \dot{w} = t'(\vec{\dot{x}})}).
\]
Since $\PA \vdash \exists w(\xi(\vec{v}) \land z_k = \mathsf{s}(w) \land w = t'(\vec{x})) \leftrightarrow (\xi(\vec{v}) \land z_k = \mathsf{s}(t'(\vec{x})))$, 
we obtain 
\[
	S \vdash \xi(\vec{v}) \land x_k = \mathsf{s}(t'(\vec{x})) \to \Phi(\gdl{\xi(\vec{\dot{v}}) \land \dot{z}_k = \mathsf{s}(t'(\vec{\dot{x}}))}) 
\]
by Lemmas \ref{MTL3} and \ref{MTL1}. 
\end{proof}

Notice that each atomic formula $t_0 = t_1$ is equivalent to $\exists z (z = t_0 \land z = t_1)$, and each negated atomic formula $t_0 \neq t_1$ is $\PA$-equivalent to $\exists z_0 \exists z_1 (t_0 + \mathsf{s}(z_0) = t_1 \lor t_1 + \mathsf{s}(z_1) = t_0)$. 
Then we obtain the following lemma. 

\begin{lem}\label{MTL}
For any quantifier-free formula $\xi(\vec{x})$, there exists a quantifier-free formula $\delta(\vec{x}, \vec{y})$ satisfying the following conditions: 
\begin{enumerate}
	\item $\PA \vdash \forall \vec{x}(\xi(\vec{x}) \leftrightarrow \exists \vec{y} \delta(\vec{x}, \vec{y}))$. 
	\item $\delta(\vec{x}, \vec{y})$ is of the form $\delta_0(\vec{x}, \vec{y}) \lor \cdots \lor \delta_{k}(\vec{x}, \vec{y})$ and each disjunct $\delta_i(\vec{x}, \vec{y})$ is of the form 
\[
	\bigwedge_{j \leq l_i} \left(z_{i, j} = t_{i, j}(\vec{x}, \vec{y}) \right)
\]
for some terms $t_{i, 0}(\vec{x}, \vec{y}), \ldots, t_{i, l_i}(\vec{x}, \vec{y})$ and variables $z_{i, 0}, \ldots, z_{i, l_i} \in \vec{x}, \vec{y}$. 
\end{enumerate}
\end{lem}

Also in our proof of Theorem \ref{MT}, we use the following $\PA$-provable form of the MRDP theorem. 

\begin{thm}[The MRDP theorem (see \cite{Kay91})]\label{MRDP}
For any $\Sigma_1$ formula $\varphi(\vec{x})$, there exists a quantifier-free formula $\delta(\vec{x}, \vec{y})$ such that $\PA \vdash \forall \vec{x}(\varphi(\vec{x}) \leftrightarrow \exists \vec{y} \delta(\vec{x}, \vec{y}))$. 
\end{thm}

\begin{proof}[Proof of Theorem \ref{MT}]
Let $\sigma(\vec{x})$ be any $\Sigma_1$ formula. 
We would like to prove $S \vdash \forall \vec{x}(\sigma(\vec{x}) \to \Phi(\gdl{\sigma(\vec{\dot{x}})}))$. 
By the MRDP theorem (Theorem \ref{MRDP}), there exists a quantifier-free formula $\delta(\vec{x}, \vec{y})$ such that $\PA \vdash \forall \vec{x}(\sigma(\vec{x}) \leftrightarrow \exists \vec{y} \delta(\vec{x}, \vec{y}))$. 
By Lemma \ref{MTL}, we may assume that $\delta(\vec{x}, \vec{y})$ is of the form indicated in the statement of Lemma \ref{MTL}. 
For each $i \leq k$, by Lemma \ref{MTL5}, we obtain
\[
	S \vdash \bigwedge_{j \leq l_i} (z_{i, j} = t_{i, j}(\vec{x}, \vec{y})) \to \Phi \left(\gdl{\bigwedge_{j \leq l_i} (\dot{z}_{i, j} = t_{i, j}(\vec{\dot{x}}, \vec{\dot{y}}))} \right). 
\]
This means
\begin{align}\label{F4}
	S \vdash \delta_i(\vec{x}, \vec{y}) \to \Phi(\gdl{\delta_i(\vec{\dot{x}}, \vec{\dot{y}})}). 
\end{align}
Since $\PA \vdash \delta_i(\vec{x}, \vec{y}) \to \delta(\vec{x}, \vec{y})$, $S \vdash \Phi(\gdl{\delta_i(\vec{\dot{x}}, \vec{\dot{y}})}) \to \Phi(\gdl{\delta(\vec{\dot{x}}, \vec{\dot{y}})})$ by $\BDU{2}$. 
Therefore by (\ref{F4}), $S \vdash \delta_i(\vec{x}, \vec{y}) \to \Phi(\gdl{\delta(\vec{\dot{x}}, \vec{\dot{y}})})$. 
Since $i \leq k$ is arbitrary, we have $S \vdash \delta_0(\vec{x}, \vec{y}) \lor \cdots \lor \delta_k(\vec{x}, \vec{y}) \to \Phi(\gdl{\delta(\vec{\dot{x}}, \vec{\dot{y}})})$. 
It follows $S \vdash \delta(\vec{x}, \vec{y}) \to \Phi(\gdl{\delta(\vec{\dot{x}}, \vec{\dot{y}})})$. 
By Lemmas \ref{MTL3} and \ref{MTL1}, we conclude $S \vdash \sigma(\vec{x}) \to  \Phi(\gdl{\sigma(\vec{\dot{x}})})$. 
\end{proof}

\section{Witnesses for non-implications}\label{Sec:W}

In this section, we exhibit examples of formulas $\Phi(x)$ satisfying and not satisfying certain conditions. 
From these examples, several non-implications between conditions are concluded. 

Our first two propositions give examples of formulas which do not satisfy $\D{1}$. 
Proofs are easy and we omit them. 

\begin{prop}\label{exQ}
Let $\PR_\mathsf{Q}(x)$ be the provability predicate of Robinson's arithmetic $\mathsf{Q}$. 
\begin{enumerate}
	\item $\PR_\mathsf{Q}(x)$ satisfies $\DG{2}$, $\SCG$, $\CB$ and $\PCG$. 
	\item $\PR_\mathsf{Q}(x)$ satisfies neither $\D{1}$ nor $\BD{2}$. 
	\item $\PA \vdash \mathsf{Con}_{\PR_\mathsf{Q}}^H$. 
\end{enumerate}
\end{prop}

\begin{prop}\label{exN}
Let $\Psi(x) : \equiv x \neq x$. 
\begin{enumerate}
	\item $\Psi(x)$ satisfies $\DG{2}$, $\DG{3}$, $\BDU{2}$ and $\CB$. 
	\item $\Psi(x)$ does not satisfy any of $\D{1}$, $\DC$ and $\PC$. 
	\item $\PA \vdash \mathsf{Con}_{\Psi}^H$. 
\end{enumerate}
\end{prop}

Feferman \cite{Fef60} proved there exists a $\Pi_1$ numeration $\pi(v)$ of $T$ in $T$ such that $\mathsf{Con}_{\PR_\pi}^H$ is provable in $\PA$. 

\begin{fact}[Feferman \cite{Fef60}]\label{exFef}
Suppose $S = T$. 
\begin{enumerate}
	\item $\PR_{\pi}(x)$ is a $\Sigma_2$ provability predicate satisfying $\DU{1}$, $\DG{2}$, $\BDU{2}$, $\SCG$, $\CB$ and $\PCG$.  
	\item $\PR_{\pi}(x)$ does not satisfy $\D{3}$.
	\item $\PA \vdash \mathsf{Con}_{\PR_{\pi}}^H$. 
\end{enumerate} 
\end{fact}

Mostowski (p.~24 in \cite{Mos65}) introduced the formula $\PR_T^M(x) :\equiv \exists y(\Prf_T(x, y) \land \neg \Prf_T(\gdl{0 \neq 0}, y))$ as an example of a $\Sigma_1$ provability predicate for which the second incompleteness theorem does not hold. 
Notice that $\PR_T^M(x)$ is $\PA$-provably equivalent to $\Pr_T(x) \land x \neq \gdl{0 \neq 0}$ because $\PA \vdash \forall x_0 \forall x_1 \forall y(\Prf_T(x_0, y) \land \Prf_T(x_1, y) \to x_0 = x_1$). 
The following proposition shows the situation for $\PR_T^M(x)$. 

\begin{prop}\label{exMos}\leavevmode
\begin{enumerate}
	\item $\PR_T^M(x)$ is a $\Sigma_1$ provability predicate satisfying $\DU{1}$, $\SCG$ and $\PCG$.  
	\item $\PR_T^M(x)$ does not satisfy any of $\D{2}$, $\BD{2}$ and $\CB$.
	\item $\PA \vdash \mathsf{Con}_{\PR_T^M}^L$ and $T \nvdash \mathsf{Con}_{\PR_T^M}^H$. 
\end{enumerate} 
\end{prop}

The existence of Rosser provability predicates satisfying some derivability conditions were discussed by Bernardi and Montagna \cite{BM84} and Arai \cite{Ara90}. 
They proved that there exists a Rosser provability predicate satisfying $\DG{2}$. 
Also Arai proved the existence of a Rosser provability predicate satisfying $\DG{3}$. 
Strictly speaking, in Arai's arguments, formulas are assumed to be in negation normal form (see \cite{Ara90}). 
We fix a natural algorithm calculating a negation normal form $\mathsf{nnf}(\varphi)$ of each formula $\varphi$ satisfying $\mathsf{nnf}(\neg \neg \varphi) \equiv \mathsf{nnf}(\varphi)$. 
Then we can understand that Arai's Rosser provability predicates $\PR^A(x)$ are of the form $\exists y (\Prf(\mathsf{nnf}(x), y) \land \forall z \leq y \neg \Prf(\mathsf{nnf}(\dot{\neg} x), z))$ for some suitable proof predicate $\Prf(x, y)$. 
Then $\PA \vdash \mathsf{Con}_{\PR^A}^H$ always holds. 
Summarizing this observation, Arai's results are stated as follows. 

\begin{fact}[Arai \cite{Ara90}]\label{exAra}
There exist $\Sigma_1$ provability predicates $\PR^A_1(x)$ and $\PR^A_2(x)$ of $T$ with: 
\begin{enumerate}
	\item  $\PR^A_1(x)$ satisfies $\D{1}$, $\DG{2}$ and $\PA \vdash \mathsf{Con}_{\PR^A_1}^H$. 
	\item  $\PR^A_2(x)$ satisfies $\D{1}$, $\DG{3}$ and $\PA \vdash \mathsf{Con}_{\PR^A_2}^H$. 
\end{enumerate}
\end{fact}
By Proposition \ref{LP1}.4, $\PR^A_1(x)$ satisfies $\BD{2}$. 
By Theorems \ref{G2} and \ref{MT}, and Propositions \ref{LP1}, \ref{UP1} and \ref{UP2}, $\PR^A_1(x)$ does not satisfy any of $\DU{1}$, $\CB$, $\BDU{2}$, $\D{3}$ and $\PC$. 
By Theorems \ref{G2-2}, \ref{Jer} and \ref{G2-3} and Proposition \ref{LP1}.4, $\PR^A_2(x)$ does not satisfy any of $\D{2}$, $\BD{2}$, $\SC$ and $\PC$. 

In \cite{Kur2}, the author proved the existence of usual Rosser provability predicates satisfying additional derivability conditions. 
That is to say, 

\begin{fact}[Kurahashi \cite{Kur2}]\label{exKur}\leavevmode
Suppose $S = T$. 
There exist $\Sigma_1$ provability predicates $\PR_1^R(x)$, $\PR_2^R(x)$ and $\PR_3^R(x)$ of $T$ with: 
\begin{enumerate}
	\item $\PR_1^R(x)$ satisfies $\D{1}$, $\DG{2}$, $\DCG$ and $\PA \vdash \mathsf{Con}_{\PR_1^R}^H$. 
	\item $\PR_2^R(x)$ satisfies $\DU{1}$, $\CB$, $\D{2}$, $\DCG$ and $\PA \vdash \mathsf{Con}_{\PR_2^R}^L$. 
	\item $\PR_3^R(x)$ satisfies $\DU{1}$, $\CB$, $\BD{2}$, $\DG{3}$, $\DCG$ and $\PA \vdash \mathsf{Con}_{\PR_3^R}^L$, but does not satisfy $\SC$. 
\end{enumerate}
\end{fact}
As in Fact \ref{exAra}.1, $\PR_1^R(x)$ satisfies $\BD{2}$, but does not satisfy any of $\DU{1}$, $\CB$, $\BDU{2}$, $\D{3}$ and $\PC$. 
By Proposition \ref{LP1}.4, $\PR_2^R(x)$ satisfies $\BD{2}$, but does not satisfy any of $\DU{2}$, $\D{3}$, $\BDU{2}$ and $\PC$ by Theorems \ref{G2} and \ref{MT}, and Propositions \ref{LP1}.6 and \ref{UP1}.3. 
By Theorems \ref{G2} and \ref{MT} and Proposition \ref{LP1}, $\PR_3^R(x)$ does not satisfy any of $\D{2}$, $\BDU{2}$ and $\PC$. 

In the remainder of this section, we introduce seven $\Sigma_1$ provability predicates $\PR_T^\mathrm{I}(x)$, $\PR_T^\mathrm{II}(x)$, $\PR_T^\mathrm{III}(x)$, $\PR_T^\mathrm{IV}(x)$, $\PR_T^\mathrm{V}(x)$, $\PR_T^\mathrm{VI}(x)$ and $\PR^\ast(x)$ which indicate several non-implications of the conditions. 
The first three provability predicates are constructed in a similar way. 
Before introducing them, we prepare a definition and a lemma. 

\begin{defn}
Let $\delta(x, z)$ be a $\Delta_1$ formula. 
\begin{enumerate}
	\item $\Prf_T[\delta](x, y) : \equiv \Prf_T(x, y) \land \forall z < y (\Prf_T(\gdl{0 \neq 0}, z) \to \delta(x, z))$. 
	\item $\PR_T[\delta](x) : \equiv \exists y \Prf_T[\delta](x, y)$. 
\end{enumerate}
\end{defn}

\begin{lem}\label{WL1}
For any $\Delta_1$ formula $\delta(x, z)$, 
\begin{enumerate}
	\item $\PR_T[\delta](x)$ is a $\Sigma_1$ provability predicate of $T$. 
	\item $\PA \vdash \forall x (\forall z(\Prf_T(\gdl{0 \neq 0}, z) \to \delta(x, z)) \to (\PR_T(x) \leftrightarrow \PR_T[\delta](x)))$. 
	\item If $\PA \vdash \forall x \forall z(\Fml(x) \land x \leq z \to \delta(x, z))$, then
\[
	\PA \vdash \forall x \forall z(\Prf_T(\gdl{0 \neq 0}, z) \land \Fml(x) \land \PR_T[\delta](x) \to \delta(x, z)).
\] 
\end{enumerate}
\end{lem}
\begin{proof}
1. Let $\varphi$ be any formula and let $n$ be any natural number. 
Since $\PA \vdash \forall z < \overline{n} \neg \Prf_T(\gdl{0 \neq 0}, z)$, $\PA \vdash \Prf_T(\gdl{\varphi}, \overline{n}) \leftrightarrow \Prf_T[\delta](\gdl{\varphi}, \overline{n})$. 
Since this equivalence is true in the standard model of arithmetic, we obtain that $\PA \vdash \PR_T(\gdl{\varphi})$ if and only if $\PA \vdash \PR_T[\delta](\gdl{\varphi})$. 
It follows that $\PR_T[\delta](x)$ is also a $\Sigma_1$ provability predicate of $T$. 

2. This is immediate from the definition. 

3. Suppose $\PA \vdash \forall x \forall z(\Fml(x) \land x \leq z \to \delta(x, z))$. 
By the definition of $\Prf_T[\delta](x, y)$, 
\begin{align}\label{eql-1}
	\PA \vdash \forall x \forall y \forall z(\Prf_T(\gdl{0 \neq 0}, z) \land \Prf_T[\delta](x, y) \land z < y \to \delta(x, z)).
\end{align}
Since $\PA \vdash \Prf_T[\delta](x, y) \to \Prf_T(x, y)$ and $\PA \vdash \Prf_T(x, y) \to x \leq y$, we have $\PA \vdash \Prf_T[\delta](x, y) \to x \leq y$. 
Thus $\PA \vdash \Prf_T[\delta](x, y) \land y \leq z \to x \leq z$. 
By the supposition, $\PA \vdash \Fml(x) \land \Prf_T[\delta](x, y) \land y \leq z \to \delta(x, z)$. 
From this with (\ref{eql-1}), we obtain
\[
	\PA \vdash \forall x \forall y \forall z (\Prf_T(\gdl{0 \neq 0}, z) \land \Fml(x) \land \Prf_T[\delta](x, y) \to \delta(x, z)), 
\]
and hence 
\[
	\PA \vdash \forall x \forall z (\Prf_T(\gdl{0 \neq 0}, z) \land \Fml(x) \land \PR_T[\delta](x) \to \delta(x, z)). 
\]
\end{proof}

Let $\Even(x)$ be a natural $\Delta_1$ formula saying that ``$x$ is the G\"odel number of a formula containing an even number of logical symbols''. 
Proposition \ref{WP1} shows that full local derivability conditions do not imply uniform derivability conditions.

\begin{prop}\label{WP1}
There exists a $\Sigma_1$ provability predicate $\PR_T^\mathrm{I}(x)$ of $T$ with: 
\begin{enumerate}
	\item $\PR_T^\mathrm{I}(x)$ satisfies $\D{1}$, $\D{2}$ and $\SC$. 
	\item $\PR_T^\mathrm{I}(x)$ does not satisfy any of $\DU{1}$, $\DU{2}$, $\DU{3}$, $\DCU$ and $\PCU$. 
\end{enumerate}
\end{prop}
\begin{proof}
Let $\PR_T^\mathrm{I}(x) : \equiv \PR_T[x \leq z \lor \Even(x)](x)$. 
Then $\PR_T^\mathrm{I}(x)$ is a $\Sigma_1$ provability predicate of $T$ by Lemma \ref{WL1}.1. 
If $\PR_T^\mathrm{I}(x)$ contains an even number of logical symbols, we replace $\PR_T^\mathrm{I}(x)$ with $\PR_T^\mathrm{I}(x) \land 0 = 0$. 
Then $\PR_T^\mathrm{I}(x)$ contains an odd number of logical symbols, and hence $\PA \vdash \forall x\neg \Even(\gdl{\PR_T^\mathrm{I}(\dot{x})})$. 

Let $\varphi$ be any formula. 
Since $\PA \vdash \forall z(\Prf_T(\gdl{0\neq 0}, z) \to \gdl{\varphi} \leq z \lor \Even(\gdl{\varphi}))$, we have $\PA \vdash \PR_T(\gdl{\varphi}) \leftrightarrow \PR_T^\mathrm{I}(\gdl{\varphi})$ by Lemma \ref{WL1}.2. 
Therefore local derivability conditions for $\PR_T^\mathrm{I}(x)$ are inherited from those for $\PR_T(x)$. 

We prove that $\PR_T^\mathrm{I}(x)$ does not satisfy any of uniform derivability conditions. 
Since $\PA \vdash \forall x \forall z(\Fml(x) \land x \leq z \to (x \leq z \lor \Even(x)))$, 
\[
	\PA \vdash \forall x \forall z (\Prf_T(\gdl{0 \neq 0}, z) \land \Fml(x) \land \PR_T^\mathrm{I}(x) \to (x \leq z \lor \Even(x)))
\]
by Lemma \ref{WL1}.3. 
For the sake of simplicity, we deal with formulas whose only free variable is $x$. 
Let $\varphi(x)$ be such a formula. 
Then
\[
	\PA \vdash \forall x \forall z (\Prf_T(\gdl{0 \neq 0}, z) \land \PR_T^\mathrm{I}(\gdl{\varphi(\dot{x})}) \to (\gdl{\varphi(\dot{x})} \leq z \lor \Even(\gdl{\varphi(\dot{x})}))). 
\]
Since $\PA \vdash x \leq \gdl{\varphi(\dot{x})}$, we obtain
\begin{align}\label{eq1-1}
	\PA \vdash \forall x \forall z (\Prf_T(\gdl{0 \neq 0}, z) \land \PR_T^\mathrm{I}(\gdl{\varphi(\dot{x})}) \to (x \leq z \lor \Even(\gdl{\varphi(\dot{x})}))). 
\end{align}

\begin{itemize}
	\item Since $\PA \vdash \forall x \neg \Even(\gdl{0=0 \land \dot{x} = \dot{x}})$, 
\[
	\PA \vdash \forall x \forall z (\Prf_T(\gdl{0 \neq 0}, z) \to (x \leq z \lor \neg \PR_T^\mathrm{I}(\gdl{0 = 0 \land \dot{x} = \dot{x}}))) 
\]
by (\ref{eq1-1}). 
Hence $\PA \vdash \PR_T(\gdl{0 \neq 0}) \to \exists x \neg \PR_T^\mathrm{I}(\gdl{0 = 0 \land \dot{x} = \dot{x}})$ because $\PA \vdash \forall z \exists x (x > z)$. 
It follows $S \nvdash \forall x \PR_T^\mathrm{I}(\gdl{0 = 0 \land \dot{x} = \dot{x}})$ because $S \nvdash \neg \PR_T(\gdl{0\neq 0})$. 
This shows that $\PR_T^\mathrm{I}(x)$ does not satisfy $\DU{1}$. 

	\item Let $\varphi(x)$ and $\psi(x)$ be formulas with $\PA \vdash \forall x \Even(\gdl{\varphi(\dot{x})}) \land \forall x \neg \Even(\gdl{\psi(\dot{x})})$. 
Then $\PA \vdash \forall x \Even(\gdl{\varphi(\dot{x}) \to \psi(\dot{x})})$. 
Since $\PA \vdash \PR_T(\gdl{0\neq 0}) \to \PR_T(\gdl{\varphi(\dot{x}) \to \psi(\dot{x})}) \land \PR_T(\gdl{\varphi(\dot{x})})$, we have 
\[
	\PA \vdash \PR_T(\gdl{0 \neq 0}) \to \PR_T^\mathrm{I}(\gdl{\varphi(\dot{x}) \to \psi(\dot{x})}) \land \PR_T^\mathrm{I}(\gdl{\varphi(\dot{x})})
\]
by the choice of $\varphi(x)$ and $\psi(x)$, and the definition of $\Prf_T^\mathrm{I}(x, y)$. 
Suppose, towards a contradiction, that $\PR_T^\mathrm{I}(x)$ satisfies $\DU{2}$, then $S \vdash \PR_T(\gdl{0 \neq 0}) \to \PR_T^\mathrm{I}(\gdl{\psi(\dot{x})})$. 
By (\ref{eq1-1}), $S \vdash \Prf_T(\gdl{0 \neq 0}, z) \to (x \leq z \lor \Even(\gdl{\psi(\dot{x})}))$, and hence $S \vdash \PR_T(\gdl{0 \neq 0}) \to \exists x \Even(\gdl{\psi(\dot{x})})$. 
By the choice of $\psi(x)$, we obtain $S \vdash \neg \PR_T(\gdl{0 \neq 0})$. 
This is a contradiction. 
Therefore $\DU{2}$ does not hold for $\PR_T^\mathrm{I}(x)$. 

	\item Let $\varphi(x)$ be a formula with $\PA \vdash \forall x \Even(\gdl{\varphi(\dot{x})})$. 
Then $\PA \vdash \PR_T(\gdl{0 \neq 0}) \to \PR_T^\mathrm{I}(\gdl{\varphi(\dot{x})})$ as described above. 
Suppose that $\DU{3}$ holds for $\PR_T^\mathrm{I}(x)$. 
Then $S \vdash \PR_T(\gdl{0 \neq 0}) \to \PR_T^\mathrm{I}(\gdl{\PR_T^\mathrm{I}(\gdl{\varphi(\dot{x})})})$. 
By (\ref{eq1-1}), we have $S \vdash \PR_T(\gdl{0 \neq 0}) \to \exists x \Even(\gdl{\PR_T^\mathrm{I}(\gdl{\varphi(\dot{x})})})$. 
Since $\PR_T^\mathrm{I}(x)$ contains an odd number of logical symbols, $\neg \PR_T(\gdl{0 \neq 0})$ is proved in $S$, and this is a contradiction. 
Hence $\DU{3}$ does not hold for $\PR_T^\mathrm{I}(x)$. 

	\item As described above, $\PA \vdash \PR_T(\gdl{0 \neq 0}) \to \exists x \neg \PR_T^\mathrm{I}(\gdl{0 = 0 \land \dot{x} = \dot{x}})$. 
If $S \vdash \forall x(0 = 0 \land x = x \to \PR_T^\mathrm{I}(\gdl{0 = 0 \land \dot{x} = \dot{x}}))$, then $S \vdash \PR_T(\gdl{0 \neq 0}) \to \exists x \neg (0 = 0 \land x = x)$. 
This implies $S \vdash \neg \PR_T(\gdl{0 \neq 0})$, a contradiction.  
Therefore $S \nvdash \forall x(0 = 0 \land x = x \to \PR_T^\mathrm{I}(\gdl{0 = 0 \land \dot{x} = \dot{x}}))$. 
This shows that $\DCU$ does not hold for $\PR_T^\mathrm{I}(x)$. 

	\item $\PCU$ fails to hold because $\PA \vdash \forall x \PRL(\gdl{0 = 0 \land \dot{x} = \dot{x}})$. 
\end{itemize}
\end{proof}

By Proposition \ref{LP1}, $\PR_T^\mathrm{I}(x)$ satisfies $\BD{2}$, $\D{3}$ and $\PC$. 
Propositions \ref{UP1}.1 and \ref{UP2}.1 imply that $\PR_T^\mathrm{I}(x)$ satisfies neither $\BDU{2}$ nor $\CB$. 

Next we prove that full uniform derivability conditions do not imply any of global derivability conditions except for $\DG{3}$, and that full derivability conditions are not sufficient for the unprovability of $\Con^{\Sigma_1}$ even if $\Phi \in \Sigma_1$. 

\begin{prop}\label{WP2}
There exists a $\Sigma_1$ provability predicate $\PR_T^\mathrm{II}(x)$ of $T$ with: 
\begin{enumerate}
	\item $\PR_T^\mathrm{II}(x)$ satisfies $\DU{1}$, $\DU{2}$, and $\SCU$. 
	\item $\PR_T^\mathrm{II}(x)$ does not satisfy any of $\DG{2}$, $\DCG$ and $\PCG$. 
	\item $\PA \vdash \mathsf{Con}_{\PR_T^\mathrm{II}}^{\Sigma_1}$. 
\end{enumerate}
\end{prop}
\begin{proof}
For each formula $\varphi$, let $n(\varphi)$ be the number of occurrences of the symbol $\neg$ in $\varphi$. 
We may use a function symbol $n(x)$ corresponding to this function such that $\PA \vdash \forall x(\Fml(x) \to n(x) \leq x)$. 

Let $\PR_T^\mathrm{II}(x)$ be the $\Sigma_1$ formula $\PR_T[n(x) \leq z \lor \Even(x)](x)$. 
Then $\PR_T^\mathrm{II}(x)$ is a $\Sigma_1$ provability predicate of $T$ by Lemma \ref{WL1}.1. 
Let $\varphi(\vec{x})$ be any formula. 
Then $\PA \vdash \forall \vec{x} (n(\gdl{\varphi(\vec{\dot{x}})}) = \overline{k})$ for some natural number $k$. 
Since $\PA \vdash \forall z(\Prf_T(\gdl{0\neq 0}, z) \to n(\gdl{\varphi(\vec{\dot{x}})}) \leq z \lor \Even(\gdl{\varphi(\vec{\dot{x}})}))$, we obtain $\PA \vdash \forall \vec{x}(\PR_T(\gdl{\varphi(\vec{\dot{x}})}) \leftrightarrow \PR_T^\mathrm{II}(\gdl{\varphi(\vec{\dot{x}})}))$ by Lemma \ref{WL1}.2. 
Therefore $\PR_T^\mathrm{II}(x)$ satisfies $\DU{1}$, $\DU{2}$ and $\SCU$.  

By Lemma \ref{WL1}.3, we have
\begin{align}\label{eq2-1}
	\PA \vdash \forall x \forall z (\Prf_T(\gdl{0 \neq 0}, z) \land \Fml(x) \land \PR_T^\mathrm{II}(x) \to (n(x) \leq z \lor \Even(x)))
\end{align}
because $\PA \vdash \forall x(\Fml(x) \land x \leq z \to n(x) \leq z \lor \Even(x))$. 

As in Proposition \ref{WP1}, failure of $\DG{2}$, $\DCG$ and $\PCG$ for $\PR_T^\mathrm{II}(x)$ follow from (\ref{eq2-1}) and the facts $\PA \vdash \forall z \exists y (\Fml(y) \land n(y) > z \land \neg \Even(y))$, $\PA \vdash \forall z \exists y (\mathsf{True}_{\Delta_0}(y) \land n(y) > z \land \neg \Even(y))$ and $\PA \vdash \forall z \exists y (\PRL(y) \land n(y) > z \land \neg \Even(y))$, respectively. 

We prove $\PA \vdash \mathsf{Con}_{\PR_T^\mathrm{II}}^{\Sigma_1}$. 
By (\ref{eq2-1}) and $\PA \vdash \forall z \exists x (\Sigma_1(x) \land \Sent(x) \land n(x) > z \land \neg \Even(x))$, we have
\[
	\PA \vdash \forall z (\Prf_T(\gdl{0 \neq 0}, z) \to \exists x (\Sigma_1(x) \land \Sent(x) \land \neg \PR_T^\mathrm{II}(x))). 
\]
It follows $\PA \vdash \PR_T^\mathrm{II}(\gdl{0\neq 0}) \to \mathsf{Con}_{\PR_T^\mathrm{II}}^{\Sigma_1}$. 
On the other hand, obviously $\PA \vdash \neg \PR_T^\mathrm{II}(\gdl{0 \neq 0}) \to \mathsf{Con}_{\PR_T^\mathrm{II}}^{\Sigma_1}$. 
Therefore we conclude $\PA \vdash \mathsf{Con}_{\PR_T^\mathrm{II}}^{\Sigma_1}$. 
\end{proof}
From Propositions \ref{UP1} and \ref{UP2}, $\PR_T^\mathrm{II}(x)$ satisfies $\BDU{2}$, $\CB$ and $\PCU$. 
By Theorem \ref{G2}, $T \nvdash \mathsf{Con}_{\PR_T^\mathrm{II}}^L$. 

We prove that the conditions $\Phi \in \Sigma_1$, $\DU{1}$, $\DG{2}$ and $\SCG$ are not sufficient for the unprovability of G\"odel's consistency statement $\Con^G$. 

\begin{prop}\label{WP3}
There exists a $\Sigma_1$ provability predicate $\PR_T^\mathrm{III}(x)$ of $T$ with: 
\begin{enumerate}
	\item $\PR_T^\mathrm{III}(x)$ satisfies $\DU{1}$, $\DG{2}$ and $\SCG$. 
	\item $\PA \vdash \mathsf{Con}_{\PR_T^\mathrm{III}}^G$. 
\end{enumerate}
\end{prop}
\begin{proof}
Let $\PR_T^\mathrm{III}(x)$ be the formula $\PR_T[\Sigma_z(x)](x)$. 
Then by Lemma \ref{WL1}.1, $\PR_T^\mathrm{III}(x)$ is a $\Sigma_1$ provability predicate of $T$. 
For any formula $\varphi(\vec{x})$, we have $\PA \vdash \forall z \forall \vec{x} (\Prf_T(\gdl{0 \neq 0}, z) \to \Sigma_z(\gdl{\varphi(\vec{\dot{x}})}))$ because $\PA \vdash \forall z \geq \overline{k} \Sigma_z(\gdl{\varphi(\vec{\dot{x}})})$ for some natural number $k$. 
Hence $\PA \vdash \PR_T(\gdl{\varphi(\vec{\dot{x}})}) \leftrightarrow \PR_T^\mathrm{III}(\gdl{\varphi(\vec{\dot{x}})})$ by Lemma \ref{WL1}.2. 
Thus $\DU{1}$ holds for $\PR_T^\mathrm{III}(x)$. 

Since $\PA \vdash \forall x \forall z(\Fml(x) \land x \leq z \to \Sigma_z(x))$, we have
\begin{align}\label{eq3-1}
	\PA \vdash \forall x \forall z(\Prf_T(\gdl{0 \neq 0}, z) \land \Fml(x) \land \PR_T^\mathrm{III}(x) \to \Sigma_z(x)) 
\end{align}
by Lemma \ref{WL1}.3. 
Then 
\[
	\PA \vdash \Fml(x) \land \Fml(y) \land \PR_T^\mathrm{III}(x \dot{\to} y) \to (\Prf_T(\gdl{0 \neq 0}, z) \to \Sigma_z(x \dot{\to} y)).
\] 
Thus 
\[
	\PA \vdash \Fml(x) \land \Fml(y) \land \PR_T^\mathrm{III}(x \dot{\to} y) \to \forall z(\Prf_T(\gdl{0 \neq 0}, z) \to \Sigma_z(y)).
\] 
By Lemma \ref{WL1}.2, 
\begin{align}\label{eq3-2}
	\PA \vdash \Fml(x) \land \Fml(y) \land \PR_T^\mathrm{III}(x \dot{\to} y) \to (\PR_T(y) \leftrightarrow \PR_T^\mathrm{III}(y)).
\end{align}
Since $\PA \vdash \PR_T^\mathrm{III}(x \dot{\to} y) \land \PR_T^\mathrm{III}(x) \to \PR_T(x \dot{\to} y) \land \PR_T(x)$, we have 
\[
	\PA \vdash \Fml(x) \land \Fml(y) \land \PR_T^\mathrm{III}(x \dot{\to} y) \land \PR_T^\mathrm{III}(x) \to  \PR_T(y)
\]
by $\DG{2}$ for $\PR_T(x)$. 
From this with (\ref{eq3-2}), 
\[
	\PA \vdash \Fml(x) \land \Fml(y) \land \PR_T^\mathrm{III}(x \dot{\to} y) \land \PR_T^\mathrm{III}(x) \to  \PR_T^\mathrm{III}(y).
\] 
This means $\DG{2}$ holds for $\PR_T^\mathrm{III}(x)$. 

Since $\PA \vdash \mathsf{True}_{\Sigma_1}(x) \to \Sigma_1(x)$, $\PA \vdash \mathsf{True}_{\Sigma_1}(x) \to (\Prf_T(\gdl{0 \neq 0}, z) \to \Sigma_z(x))$. 
By Lemma \ref{WL1}.2, $\PA \vdash \mathsf{True}_{\Sigma_1}(x) \to (\PR_T(x) \leftrightarrow \PR_T^\mathrm{III}(x))$.
By $\SCG$ for $\PR_T(x)$, we obtain $\PA \vdash \mathsf{True}_{\Sigma_1}(x) \to \PR_T^\mathrm{III}(x)$. 

By (\ref{eq3-1}) and $\PA \vdash \forall z \exists x(\Fml(x) \land \neg \Sigma_z(x))$, we have $\PA \vdash \PR_T(\gdl{0 \neq 0}) \to \exists x (\Fml(x) \land \neg \PR_T^\mathrm{III}(x))$. 
Thus $\PA \vdash \PR_T(\gdl{0 \neq 0}) \to \mathsf{Con}_{\PR_T^\mathrm{III}}^G$. 
On the other hand, since $\PA \vdash \neg \PR_T(\gdl{0 \neq 0}) \to \neg \PR_T^\mathrm{III}(\gdl{0 \neq 0})$, we have $\PA \vdash \neg \PR_T(\gdl{0 \neq 0}) \to \mathsf{Con}_{\PR_T^\mathrm{III}}^G$. 
Therefore $\PA \vdash \mathsf{Con}_{\PR_T^\mathrm{III}}^G$. 
\end{proof}
By Propositions \ref{UP1} and \ref{UP2}, $\PR_T^\mathrm{III}(x)$ satisfies $\BDU{2}$, $\CB$ and $\PCU$. 
Corollary \ref{GC1} implies that $\PCG$ fails to hold for $\PR_T^\mathrm{III}(x)$ and $T \nvdash \mathsf{Con}_{\PR_T^\mathrm{III}}^{\Sigma_1}$. 

We prove that there exists a $\Sigma_1$ provability predicate which satisfies the Hilbert--Bernays--L\"ob derivability conditions, but does not satisfy $\SC$. 
The following proof is based on the construction presented in Section 5 of Visser \cite{Vis16}. 

\begin{prop}\label{WP4}
There exists a $\Sigma_1$ provability predicate $\PR_T^\mathrm{IV}(x)$ of $T$ which satisfies $\D{1}$, $\DG{2}$ and $\DG{3}$, but does not satisfy $\SC$. 
\end{prop}
\begin{proof}
We say an $\mathcal{L}_A$-formula $\varphi$ is \textit{propositionally atomic} if it is not a Boolean combination of proper subformulas of $\varphi$. 
We fix a bijective mapping $f$ from the set of all propositional variables to the set of all propositionally atomic formulas. 
For each propositionally atomic formula $\Phi(x)$, the mapping $f$ can be extended to the mapping $f_{\Phi}$ from the set of all modal formulas to the set of all $\mathcal{L}_A$-formulas satisfying the following clauses: 
\begin{enumerate}
	\item $f_\Phi(p)$ is $f(p)$ for each propositional variable $p$; 
	\item $f_\Phi$ commutes with every propositional connective; 
	\item $f_\Phi(\Box A)$ is $\Phi(\gdl{f_\Phi(A)})$. 
\end{enumerate}

For any finite set $X$ of modal formulas and any modal formula $A$, $A$ is said to be derived in $X$ if $A$ is provable in the system whose axioms are elements of $X$ and whose inference rules are Modus Ponens $\dfrac{B\ \ B \to C}{C}$ and Necessitation $\dfrac{B}{\Box B}$. 

For each natural number $n$, let $\Th_n(T)$ be the finite set of all $\mathcal{L}_A$-formulas having a $T$-proof whose G\"odel number is less than or equal to $n$. 
We write $T \vdash_{\Phi, n} \varphi$ if there exist a finite set $X$ of modal formulas and a modal formula $A$ such that $f_\Phi(X) = \Th_n(T)$, $f_\Phi(A)$ is $\varphi$ and $A$ is derived in $X$. 
For $m < n$, $T \vdash_{\Phi, m} \varphi$ implies $T \vdash_{\Phi, n} \varphi$ because $\Th_m(T) \subseteq \Th_n(T)$. 
As shown in Visser \cite{Vis16}, the ternary relation $T \vdash_{\Phi, n} \varphi$ is computable. 
Thus we obtain a $\Delta_1$ formula $P_T(\gdl{\Phi}, x, y)$ saying that $x$ is the G\"odel number of a formula $\varphi$ satisfying $T \vdash_{\Phi, y} \varphi$. 

By the Fixed Point Lemma, there exist a $\Sigma_1$ formula $\PR_T^{\mathrm{IV}}(x)$ and a $\Sigma_1$ sentence $\sigma$ satisfying the following equivalences: 
\begin{enumerate}
	\item $P'_T(x, y) \equiv P_T(\gdl{\PR_T^{\mathrm{IV}}}, x, y)$; 
	\item $\PA \vdash \PR_T^{\mathrm{IV}}(x) \leftrightarrow \exists y(P_T'(x, y) \land \forall z < y \neg P_T'(\gdl{\neg \sigma}, z))$; 
	\item $\PA \vdash \sigma \leftrightarrow \exists z(P_T'(\gdl{\neg \sigma}, z) \land \forall y \leq z \neg P_T'(\gdl{\sigma}, y))$. 
\end{enumerate}

First, we prove $T \nvdash_{\PR_T^{\mathrm{IV}}, n} \neg \sigma$ for all $n$ by induction on $n$. 
Suppose $T \nvdash_{\PR_T^{\mathrm{IV}}, m} \neg \sigma$ for all $m < n$. 
Then $\PA \vdash \forall z < \overline{n} \neg P_T'(\gdl{\neg \sigma}, z)$. 

Let $X$ be any finite set of modal formulas with $f_{\PR_T^{\mathrm{IV}}}(X) = \Th_n(T)$. 
Let $A$ be any modal formula derived in $X$, then $T \vdash_{\PR_T^{\mathrm{IV}}, n} f_{\PR_T^{\mathrm{IV}}}(A)$. 
Hence we have $\PA \vdash P_T'(\gdl{f_{\PR_T^{\mathrm{IV}}}(A)}, \overline{n})$, and thus $\PA \vdash \PR_T^{\mathrm{IV}}(\gdl{f_{\PR_T^{\mathrm{IV}}}(A)})$. 
Moreover, we show $T \vdash f_{\PR_T^{\mathrm{IV}}}(A)$. 
This is proved by induction on the length of derivation in $X$. 
If $A \in X$, then $f_{\PR_T^{\mathrm{IV}}}(A) \in \Th_n(T)$, and $f_{\PR_T^{\mathrm{IV}}}(A)$ has a $T$-proof. 
If $A$ is derived from $B$ and $B \to A$ by Modus Ponens and $T \vdash f_{\PR_T^{\mathrm{IV}}}(B) \land f_{\PR_T^{\mathrm{IV}}}(B \to A)$, then $T \vdash f_{\PR_T^{\mathrm{IV}}}(A)$. 
If $A$ is derived from $B$ by Necessitation, then $A$ is of the form $\Box B$. 
Since $\PA \vdash \PR_T^{\mathrm{IV}}(\gdl{f_{\PR_T^{\mathrm{IV}}}(B)})$ as above, we get $\PA \vdash f_{\PR_T^{\mathrm{IV}}}(A)$. 
In this paragraph, we have shown that if $T \vdash_{\PR_T^{\mathrm{IV}}, n} \varphi$, then $T \vdash \varphi$. 

Suppose, towards a contradiction, $T \vdash_{\PR_T^{\mathrm{IV}}, n} \neg \sigma$. 
Then $T \vdash \neg \sigma$. 
Since $T \nvdash \sigma$, $T \nvdash_{\PR_T^{\mathrm{IV}}, m} \sigma$ for all $m \leq n$. 
Therefore $\PA \vdash P_T'(\gdl{\neg \sigma}, \overline{n}) \land \forall y \leq \overline{n} \neg P_T'(\gdl{\sigma}, y)$. 
By the definition of $\sigma$, we have $\PA \vdash \sigma$. 
This is a contradiction. 
We obtain $T \nvdash_{\PR_T^{\mathrm{IV}}, n} \neg \sigma$. 

If $T \vdash \varphi$, then $\varphi \in \Th_n(T)$ for some $n$. 
Then $T \vdash_{\PR_T^{\mathrm{IV}}, n} \varphi$ trivially holds, and hence $\PA \vdash P_T'(\gdl{\varphi}, \overline{n})$. 
Since $\PA \vdash \forall z < \overline{n} P_T'(\gdl{\neg \sigma}, z)$, we obtain $\PA \vdash \PR_T^{\mathrm{IV}}(\gdl{\varphi})$. 
On the other hand, we assume $\PA \vdash \PR_T^{\mathrm{IV}}(\gdl{\varphi})$. 
Then $P_T'(\gdl{\varphi}, \overline{n})$ is true in the standard model of arithmetic for some $n$. 
This means $T \vdash_{\PR_T^{\mathrm{IV}}, n} \varphi$. 
Then we obtain $T \vdash \varphi$. 
Therefore we have shown that $\PR_T^{\mathrm{IV}}(x)$ is a $\Sigma_1$ provability predicate of $T$. 

We prove $\DG{2}$ for $\PR_T^{\mathrm{IV}}(x)$. 
We work in $S$. 
Suppose $\PR_T^{\mathrm{IV}}(\gdl{\varphi})$ and $\PR_T^{\mathrm{IV}}(\gdl{\varphi \to \psi})$ are true. 
Then for some $n$, $T \vdash_{\PR_T^\mathrm{IV}, n} \varphi$, $T \vdash_{\PR_T^\mathrm{IV}, n} \varphi \to \psi$ and $T \nvdash_{\PR_T^\mathrm{IV}, m} \neg \sigma$ for all $m < n$. 
Then $T \vdash_{\PR_T^\mathrm{IV}, n} \psi$. 
Thus $\PR_T^{\mathrm{IV}}(\gdl{\psi})$ is true. 

We prove $\DG{3}$ for $\PR_T^{\mathrm{IV}}(x)$. 
We proceed in $S$. 
Suppose $\PR_T^{\mathrm{IV}}(\gdl{\varphi})$ is true. 
Then for some $n$, $T \vdash_{\PR_T^\mathrm{IV}, n} \varphi$ and $T \nvdash_{\PR_T^\mathrm{IV}, m} \neg \sigma$ for all $m < n$. 
Then $T \vdash_{\PR_T^\mathrm{IV}, n} \PR_T^{\mathrm{IV}}(\gdl{\varphi})$. 
Thus $\PR_T^{\mathrm{IV}}(\gdl{\PR_T^{\mathrm{IV}}(\gdl{\varphi})})$ is true. 

At last, we prove that $\SC$ fails to hold. 
Suppose, for a contradiction, $T \vdash \sigma \to \PR_T^{\mathrm{IV}}(\gdl{\sigma})$. 
By witness comparison argument, we have $\PA \vdash \sigma \to \neg \PR_T^{\mathrm{IV}}(\gdl{\sigma})$. 
Thus $T \vdash \neg \sigma$. 
Then $T \vdash_{\PR_T^\mathrm{IV}, n} \neg \sigma$ for some $n$. 
This is a contradiction. 
Therefore we conclude $T \nvdash \sigma \to \PR_T^{\mathrm{IV}}(\gdl{\sigma})$. 
\end{proof}
By Proposition \ref{LP1}, Theorem \ref{MT}, Proposition \ref{UP1}.3 and Proposition \ref{UP2}.1, $\PR_T^\mathrm{IV}(x)$ does not satisfy any of $\PC$, $\BDU{2}$, $\DU{1}$ and $\CB$.

The next two propositions show that $\{\D{1}, \SC\}$ and $\{\D{1}, \PC\}$ are incomparable. 

\begin{prop}\label{WP5}
There exists a $\Sigma_1$ provability predicate $\PR_T^\mathrm{V}(x)$ of $T$ which satisfies $\SCG$, but does not satisfy any of $\DU{1}$ and $\PC$. 
\end{prop}
\begin{proof}
Let $T_0$ be any finite subtheory of $T$ containing $\mathsf{Q}$ with $\bigwedge T_0$ is not a $\Pi_1$ sentence. 
Let $\Prf_T'(v, x, y)$ be the $\Delta_1$ formula 
\[
	\Prf_T(x, y) \land (\exists z < y \Prf_T(\dot{\neg} v, z) \to \Sigma_1(x)). 
\]
By the Fixed Point Lemma, there exists a $\Sigma_1$ sentence $\sigma$ satisfying
\[
	\PA \vdash \sigma \leftrightarrow \exists z(\Prf_T(\gdl{\neg \sigma}, z) \land \forall y \leq z \neg \Prf_T'(\gdl{\sigma}, \gdl{\bigwedge T_0 \to \sigma}, y)). 
\]
Let $\Prf_T^\mathrm{V}(x, y) : \equiv \Prf_T'(\gdl{\sigma}, x, y)$ and let $\PR_T^\mathrm{V}(x) : \equiv \exists y \Prf_T^\mathrm{V}(x, y)$. 
Then
\begin{itemize}
	\item $\PA \vdash \Prf_T^\mathrm{V}(x, y) \leftrightarrow \Prf_T(x, y) \land (\exists z < y \Prf_T(\gdl{\neg \sigma}, z) \to \Sigma_1(x))$. 
	\item $\PA \vdash \sigma \leftrightarrow \exists z(\Prf_T(\gdl{\neg \sigma}, z) \land \forall y \leq z \neg \Prf_T^\mathrm{V}(\gdl{\bigwedge T_0 \to \sigma}, y))$. 
\end{itemize}

First, we prove $T \nvdash \neg \sigma$. 
If $T \vdash \neg \sigma$, then for some natural number $p$, $\PA \vdash \Prf_T(\gdl{\neg \sigma}, \overline{p})$. 
Since $T \nvdash \sigma$, obviously $T \nvdash \bigwedge T_0 \to \sigma$. 
Then $\PA \vdash \forall y \leq \overline{p} \neg \Prf_T(\gdl{\bigwedge T_0 \to \sigma}, y)$. 
Since $\Prf_T^\mathrm{V}(x, y)$ implies $\Prf_T(x, y)$, we have $\PA \vdash \forall y \leq \overline{p} \neg \Prf_T^\mathrm{V}(\gdl{\bigwedge T_0 \to \sigma}, y)$. 
Then $\PA \vdash \sigma$ by the definition of $\sigma$. 
This is a contradiction. 
Therefore $T \nvdash \neg \sigma$. 

It follows that for any natural number $n$, $\PA \vdash \neg \Prf_T(\gdl{\neg \sigma}, \overline{n})$. 
Then for any formula $\varphi$, $\PA \vdash \Prf_T(\gdl{\varphi}, \overline{n}) \leftrightarrow \Prf_T^\mathrm{V}(\gdl{\varphi}, \overline{n})$. 
Thus $\PR_T^\mathrm{V}(x)$ is a $\Sigma_1$ provability predicate of $T$. 

Since $\PA \vdash \Sigma_1(x) \to (\PR_T(x) \leftrightarrow \PR_T^\mathrm{V}(x))$ by the definition, $\SCG$ for $\PR_T^\mathrm{V}(x)$ easily follows from $\SCG$ for $\PR_T(x)$. 

We prove that $\PC$ fails to hold for $\PR_T^\mathrm{V}(x)$. 
If $\PR_T^\mathrm{V}(x)$ satisfied $\PC$, then $S \vdash \PRL(\gdl{\bigwedge T_0 \to \sigma}) \to \PR_T^\mathrm{V}(\gdl{\bigwedge T_0 \to \sigma})$. 
By formalized deduction theorem, $S \vdash \PR_{[T_0]}(\gdl{\sigma}) \to \PR_T^\mathrm{V}(\gdl{\bigwedge T_0 \to \sigma})$. 
By $\SC$ for $\PR_{[T_0]}(x)$,
\begin{align}\label{eq4-1}
	S \vdash \sigma \to \PR_T^\mathrm{V}(\gdl{\bigwedge T_0 \to \sigma}).
\end{align}

By the definition of $\Prf_T^\mathrm{V}(x, y)$, we obtain
\[
	\PA \vdash \Prf_T^\mathrm{V}(\gdl{\bigwedge T_0 \to \sigma}, y) \land \Prf_T(\gdl{\neg \sigma}, z) \land z < y \to \Sigma_1(\gdl{\bigwedge T_0 \to \sigma}).
\]
Since $\bigwedge T_0 \to \sigma$ is not $\Sigma_1$, 
\[
	\PA \vdash \Prf_T^\mathrm{V}(\gdl{\bigwedge T_0 \to \sigma}, y) \land \Prf_T(\gdl{\neg \sigma}, z) \to y \leq z. 
\]
It follows
\[
	\PA \vdash \PR_T^\mathrm{V}(\gdl{\bigwedge T_0 \to \sigma}) \to \forall z(\Prf_T(\gdl{\neg \sigma}, z) \to \exists y \leq z \Prf_T^\mathrm{V}(\gdl{\bigwedge T_0 \to \sigma}, y)). 
\]
This means $\PA \vdash \PR_T^\mathrm{V}(\gdl{\bigwedge T_0 \to \sigma}) \to \neg \sigma$. 
From this with (\ref{eq4-1}), $S \vdash \sigma \to \neg \sigma$, and hence $S \vdash \neg \sigma$. 
This is a contradiction. 
Therefore $\PR_T^\mathrm{V}(x)$ does not satisfy $\PC$. 

Finally, we prove that $\PR_T^\mathrm{V}(x)$ does not satisfy $\DU{1}$. 
Let $\varphi(x)$ be any formula such that $\PA \vdash \forall x \neg \Sigma_1(\gdl{\varphi(\dot{x})})$ and $T \vdash \forall x \varphi(x)$. 
Since $\PA \vdash \Prf_T(\gdl{\varphi(\dot{z})}, y) \to z < y$, we have $\PA \vdash \PR_T^\mathrm{V}(\gdl{\varphi(\dot{z})}) \land \Prf_T(\gdl{\neg \sigma}, z) \to \Sigma_1(\gdl{\varphi(\dot{z})})$ by the definition of $\Prf_T^\mathrm{V}(x, y)$. 
Hence $\PA \vdash \PR_T^\mathrm{V}(\gdl{\varphi(\dot{z})}) \to \neg \Prf_T(\gdl{\neg \sigma}, z)$. 
Then $\PA \vdash \forall x \PR_T^\mathrm{V}(\gdl{\varphi(\dot{x})}) \to \neg \PR_T(\gdl{\neg \sigma})$. 
Since $T \nvdash \neg \PR_T(\gdl{\neg \sigma})$, we conclude that $T \nvdash \forall x \PR_T^\mathrm{V}(\gdl{\varphi(\dot{x})})$. 
\end{proof}
By Propositions \ref{LP1} and \ref{UP2}. $\PR_T^\mathrm{V}(x)$ does not satisfy any of $\D{2}$, $\BD{2}$ and $\CB$.

We give an example of Mostowski-like $\Sigma_1$ provability predicate which satisfies $\PCG$ but does not satisfy $\SC$. 

\begin{prop}\label{WP6}
There exists a $\Sigma_1$ provability predicate $\PR_T^\mathrm{VI}(x)$ of $T$ with: 
\begin{enumerate}
	\item $\PR_T^\mathrm{VI}(x)$ satisfies $\DU{1}$, $\DG{3}$, $\DCG$ and $\PCG$. 
	\item $\PR_T^\mathrm{VI}(x)$ satisfies neither $\SC$ nor $\CB$.
\end{enumerate} 
\end{prop}
\begin{proof}
Let $\xi$ be a $\Pi_1$ sentence undecidable in $T$ such as Rosser's sentence (see \cite{Lin03}), and let $\xi'$ be the sentence $\xi \lor 0 = \mathsf{s}(0)$ which is also undecidable in $T$. 
Let $\PR_T^\mathrm{VI}(x) : \equiv \PR_T(x) \land x \neq \gdl{\neg \xi'}$. 
Obviously, 
\begin{align}\label{eqS}
	\PA \vdash \forall x(x \neq \gdl{\neg \xi'} \to (\PR_T(x) \leftrightarrow \PR_T^\mathrm{VI}(x))).
\end{align}

Since $\neg \xi'$ is not provable in $T$, $\PR_T^\mathrm{VI}(x)$ is a $\Sigma_1$ provability predicate of $T$, and also $\DU{1}$ holds for $\PR_T^\mathrm{VI}(x)$. 
The conditions $\DG{3}$ and $\DCG$ follow from $\PA \vdash \forall x(\gdl{\PR_T^\mathrm{VI}(\dot{x})} \neq \gdl{\neg \xi'})$ and $\PA \vdash \forall x (\mathsf{True}_{\Delta_0}(x) \to x \neq \gdl{\neg \xi'})$, respectively. 

We prove $\PCG$. 
Let $M$ be an $\mathcal{L}_A$-structure whose domain is a singleton $\{e\}$. 
Then for every closed $\mathcal{L}_A$-term $t$, $t^M = e$. 
Thus $M \models \xi \lor 0 = \mathsf{s}(0)$. 
Therefore $\neg \xi'$ is not provable in predicate calculus. 
The above argument can be formalized in $\PA$, and so $\PA \vdash \forall x(\Fml(x) \to (\PRL(x) \to x \neq \gdl{\neg \xi'}))$. 
Then by $\PCG$ for $\PR_T(x)$, we conclude $\PA \vdash \forall x(\Fml(x) \to (\PRL(x) \to \PR_T^\mathrm{VI}(x)))$. 

Since $\PA \vdash \neg \PR_T^\mathrm{VI}(\gdl{\neg \xi'})$ and $T \nvdash \xi'$, we can prove $S \nvdash \PR_T^\mathrm{VI}(\gdl{\forall x \neg (\xi \lor x = \mathsf{s}(0))}) \to \forall x \PR_T^\mathrm{VI}(\gdl{\neg (\xi \lor \dot{x} = \mathsf{s}(0))})$ by (\ref{eqS}). 
The conditions $\SC$ and $\CB$ fail to hold because of them. 
\end{proof}
By Proposition \ref{LP1}, $\PR_T^\mathrm{VI}(x)$ satisfies neither $\D{2}$ nor $\BD{2}$.

At last, we prove that our Theorem \ref{MT} is actually an improvement of Buchholz's theorem (Theorem \ref{UBuc}). 

\begin{thm}\label{MT2}
There exists a $\Sigma_1$ provability predicate $\PR^\ast(x)$ of $\PA$ which satisfies $\DU{1}$, $\BDU{2}$, $\SCG$ and $\PCG$ but does not satisfy $\D{2}$. 
\end{thm}

This theorem is proved by using Beklemishev's arithmetical completeness theorem of the bimodal logic $\CS_2$ with respect to independent $\Sigma_1$ numerations (see Beklemishev \cite{Bek92}). 
For this, we need some preparations. 
The language of $\CS_2$ is that of propositional logic equipped with two unary modal operators $[0]$ and $[1]$. 
Formulas in this language are called $\CS_2$-formulas. 
The axioms of the bimodal logic $\CS_2$ are propositional tautologies and the formulas $[i] (p \to q) \to ([i]p \to [i]q)$, $[i] p \to [j][i] p$ and $[i]([i] p \to p) \to [i]p$ for $i, j \in \{0, 1\}$. 
The inference rules of $\CS_2$ are modus ponens $\dfrac{A, \ \ A \to B}{B}$, necessitation $\dfrac{A}{[i] A}$ for $i \in \{0, 1\}$, and uniform substitution. 

We say a structure $M = (W, K_0, K_1, \prec, \Vdash, b)$ is a \textit{$\CS_2$-model} if it satisfies the following conditions: 
\begin{enumerate}
	\item $W$ is a nonempty finite set. 
	\item $K_0$ and $K_1$ are subsets of $W$ with $W = K_0 \cup K_1$. 
	\item $\prec$ is a strict partial ordering over $W$. 
	\item $b \in K_0 \cap K_1$ and $b \prec x$ for all $x \in W \setminus \{b\}$. 
	\item $\Vdash$ is a binary relation between $W$ and the set of all $\CS_2$-formulas such that $\Vdash$ satisfies the usual conditions for satisfaction and the following condition: for $i \in \{0, 1\}$, $x \Vdash [i] A$ if and only if for all $y \in K_i$, if $x \prec y$, then $y \Vdash A$. 
\end{enumerate}
A $\CS_2$-formula $A$ is said to be \textit{true} in a $\CS_2$-model $M = (W, K_0, K_1, \prec, \Vdash, b)$ if $b \Vdash A$. 
The modal logic $\CS_2$ is sound and complete with respect to $\CS_2$ models. 

\begin{thm}[See Smory\'nski \cite{Smo85}]
For any $\CS_2$-formula $A$, the following are equivalent: 
\begin{enumerate}
	\item $\CS_2 \vdash A$. 
	\item $A$ is true in all $\CS_2$-models. 
\end{enumerate}
\end{thm}

Let $\alpha_0(v)$ and $\alpha_1(v)$ be any $\Sigma_1$ numerations of $\PA$. 
A mapping $f$ from $\CS_2$-formulas to $\mathcal{L}_A$-sentences is a \textit{$(\alpha_0, \alpha_1)$-interpretation} if $f$ commutes with each propositional connective, and $f([i] A) \equiv \PR_{\alpha_i}(\gdl{f(A)})$ for $i \in \{0, 1\}$. 
Beklemishev proved that $\CS_2$ is sound and complete with respect to this kind of interpretations. 

\begin{thm}[The arithmetical completeness theorem of $\CS_2$ (Beklemishev \cite{Bek92})]
For any $\CS_2$-formula $A$, the following are equivalent:
\begin{enumerate}
	\item $\CS_2 \vdash A$. 
	\item For any $\Sigma_1$ numerations $\alpha_0(v)$ and $\alpha_1(v)$ of $\PA$ and any $(\alpha_0, \alpha_1)$-interpretation $f$, $\PA \vdash f(A)$. 
\end{enumerate}
\end{thm}

We are ready to prove Theorem \ref{MT2}. 

\begin{proof}[Proof of Theorem \ref{MT2}]
Let us consider a $\CS_2$-model $M = (W, K_0, K_1, \prec, \Vdash, b)$ satisfying the following conditions: 
\begin{enumerate}
	\item $W = \{b, x_0, x_1\}$, 
	\item $K_0 = \{b, x_0\}$ and $K_1 = \{b, x_1\}$, 
	\item $\prec = \{(b, x_0), (b, x_1)\}$, 
	\item $x_0 \Vdash p$ and $x_1 \nVdash p$. 
\end{enumerate}

Then $b \Vdash [0]p \land [1]\neg p \land \neg [0] \bot \land \neg [1] \bot$. 
Thus $\CS_2 \nvdash [0]p \land [1] \neg p \to [0] \bot \lor [1] \bot$. 
By the arithmetical completeness theorem of $\CS_2$, there are $\Sigma_1$ numerations $\alpha_0(v)$ and $\alpha_1(v)$ of $\PA$, and a $(\alpha_0, \alpha_1)$-interpretation $f$ such that $\PA \nvdash f([0]p \land [1] \neg p \to [0] \bot \lor [1] \bot)$. 
Let $\xi : \equiv f(p)$, then 
\begin{align}\label{eq11}
	\PA \nvdash \PR_{\alpha_0}(\gdl{\xi}) \land \PR_{\alpha_1}(\gdl{\neg \xi}) \to \neg \mathsf{Con}_{\PR_{\alpha_0}} \lor \neg \mathsf{Con}_{\PR_{\alpha_1}}. 
\end{align}

Let $\PR^\ast(x)$ be the $\Sigma_1$ formula $\PR_{\alpha_0}(x) \lor \PR_{\alpha_1}(x)$. 
Then $\PR^\ast(x)$ is obviously a $\Sigma_1$ provability predicate of $\PA$. 
Moreover $\DU{1}$, $\SCG$ and $\PCG$ are inherited from $\PR_{\alpha_0}(x)$. 

First, we prove that $\PR^\ast(x)$ satisfies $\BDU{2}$. 
Suppose $\PA \vdash \forall \vec{x} (\varphi(\vec{x}) \to \psi(\vec{x}))$. 
Then since both $\PR_{\alpha_0}(x)$ and $\PR_{\alpha_1}(x)$ satisfy $\BDU{2}$, we have
\[
	\PA \vdash \PR_{\alpha_0}(\gdl{\varphi(\vec{\dot{x}})}) \to \PR_{\alpha_0}(\gdl{\psi(\vec{\dot{x}})})\ \text{and}\ \PA \vdash \PR_{\alpha_1}(\gdl{\varphi(\vec{\dot{x}})}) \to \PR_{\alpha_1}(\gdl{\psi(\vec{\dot{x}})}). 
\]
By the definition of $\PR^\ast(x)$, 
\[
	\PA \vdash \PR_{\alpha_0}(\gdl{\varphi(\vec{\dot{x}})}) \to \PR^\ast(\gdl{\psi(\vec{\dot{x}})})\ \text{and}\ \PA \vdash \PR_{\alpha_1}(\gdl{\varphi(\vec{\dot{x}})}) \to \PR^\ast(\gdl{\psi(\vec{\dot{x}})}). 
\]
Therefore we conclude
\[
	\PA \vdash \forall \vec{x} (\PR^\ast(\gdl{\varphi(\vec{\dot{x}})}) \to \PR^\ast(\gdl{\psi(\vec{\dot{x}})})). 
\]

At last, we prove that $\PR^\ast(x)$ does not satisfy $\D{2}$. 
Suppose, towards a contradiction,  
\[
	\PA \vdash \PR^\ast(\gdl{\xi \to 0 \neq 0}) \to (\PR^\ast(\gdl{\xi}) \to \PR^\ast(\gdl{0 \neq 0})). 
\]
Then by the definition of $\PR^\ast(x)$, 
\[
	\PA \vdash \PR_{\alpha_0}(\gdl{\neg \xi}) \lor \PR_{\alpha_1}(\gdl{\neg \xi}) \to (\PR_{\alpha_0}(\gdl{\xi}) \lor \PR_{\alpha_1}(\gdl{\xi}) \to \neg \mathsf{Con}_{\PR_{\alpha_0}} \lor \neg \mathsf{Con}_{\PR_{\alpha_1}}). 
\]
By logic, we obtain 
\[
	\PA \vdash \PR_{\alpha_0}(\gdl{\xi}) \land \PR_{\alpha_1}(\gdl{\neg \xi}) \to \neg \mathsf{Con}_{\PR_{\alpha_0}} \lor \neg \mathsf{Con}_{\PR_{\alpha_1}}. 
\]
This contradicts (\ref{eq11}). 
Therefore we conclude 
\[
	\PA \nvdash \PR^\ast(\gdl{\xi \to 0 \neq 0}) \to (\PR^\ast(\gdl{\xi}) \to \PR^\ast(\gdl{0 \neq 0})). 
\]
\end{proof}
By Proposition \ref{UP2}.2, $\PR^\ast(x)$ satisfies $\CB$. 

As we have seen, examples of formulas given in this section show several non-implications between conditions. 
For instance, the following non-implications related to Proposition \ref {LP1} are also obtained. 
\begin{enumerate}
	\item $\DC \not \Rightarrow \D{1}$ (Proposition \ref{exQ}). 
	\item $\{\BD{m} : m \geq 2\} \not \Rightarrow \D{1}$ (Proposition \ref{exN}). \\
	For all $m \geq 2$, $\D{1} \not \Rightarrow \BD{m}$ (Proposition \ref{exMos}). 
	\item For all $m \geq 1$, $\D{2} \not \Rightarrow \BD{m}$ (Proposition \ref{exQ}). 
	\item $\D{3} \not \Rightarrow \DC$ (Proposition \ref{exN}). 
\end{enumerate}

However, we do not have enough such non-implications between conditions including uniform and global versions. 
We close this paper with the following problem. 

\begin{prob}
Study further non-implications between derivability conditions. 
\end{prob}

\section{Acknowledgments}

This work was partly supported by JSPS KAKENHI Grant Numbers 16K17653 and 19K14586. 
The author would like to thank Toshiyasu Arai, Yong Cheng, Makoto Kikuchi, Hidenori Kurokawa, Yuya Okawa and Albert Visser for their valuable comments. 
The author would also like to thank the anonymous referee for pointing out an error in an earlier version of the manuscript.

\bibliographystyle{plain}
\bibliography{ref}

\begin{thebibliography}{10}

\bibitem{Ara90}
Toshiyasu Arai.
\newblock Derivability conditions on {R}osser's provability predicates.
\newblock {\em Notre Dame Journal of Formal Logic}, 31(4):487--497, 1990.

\bibitem{AB05}
Sergei~N. Artemov and Lev~D. Beklemishev.
\newblock {\em Provability Logic}, volume~13 of {\em Handbook of Philosophical
  Logic}, pages 189--360.
\newblock Springer, Dordrecht, 2nd edition, 2005.

\bibitem{Bek92}
Lev~D. Beklemishev.
\newblock Independent enumerations of theories and recursive progressions.
\newblock {\em Sibirskii Matematichskii Zhurnal}, 33(5):22--46, 1992.
\newblock English translation is in {\it Siberian Mathematical Journal}, 33(5),
  760-783, 1992.

\bibitem{BM84}
Claudio Bernardi and Franco Montagna.
\newblock Equivalence relations induced by extensional formulae: classification
  by means of a new fixed point property.
\newblock {\em Fundamenta Mathematicae}, 124(3):221--233, 1984.

\bibitem{Boo93}
George Boolos.
\newblock {\em The logic of provability}.
\newblock Cambridge University Press, Cambridge, 1993.

\bibitem{Buc93}
Wilfried Buchholz.
\newblock Mathematische {L}ogik {II}.
\newblock
  \url{http://www.mathematik.uni-muenchen.de/~buchholz/articles/LogikII.ps},
  1993.

\bibitem{Fef60}
Solomon Feferman.
\newblock Arithmetization of metamathematics in a general setting.
\newblock {\em Fundamenta Mathematicae}, 49:35--92, 1960.

\bibitem{Goed31}
Kurt G{\"o}del.
\newblock {\"U}ber formal unentscheidbare {S\"a}tze der {P}rincipia
  {M}athematica und verwandter {S}ysteme {I}. (in {G}erman).
\newblock {\em Monatshefte f{\"u}r Mathematik und Physik}, 38(1):173--198,
  1931.
\newblock English translation in Kurt G{\"o}del, {\it Collected Works}, Vol. 1
  (pp. 145--195).

\bibitem{HP93}
Petr H{\'a}jek and Pavel Pudl{\'a}k.
\newblock {\em Metamathematics of First-Order Arithmetic}.
\newblock Perspectives in Mathematical Logic. Springer-Verlag, Berlin, 1993.

\bibitem{HB39}
David Hilbert and Paul Bernays.
\newblock {\em Grundlagen der {M}athematik. Vol. II}.
\newblock Springer, Berlin, 1939.

\bibitem{HC}
G.~E. Hughes and M.~J. Cresswell.
\newblock {\em A new introduction to modal logic}.
\newblock Routledge, London, 1996.

\bibitem{JD98}
Giorgi Japaridze and Dick de~Jongh.
\newblock {\em The logic of provability}, volume 137 of {\em Studies in Logic
  and the Foundations of Mathematics}, pages 475--546.
\newblock North-Holland, Amsterdam, 1998.

\bibitem{Jer73}
Robert~G. Jeroslow.
\newblock Redundancies in the {H}ilbert-{B}ernays derivability conditions for
  {G\"o}del's second incompleteness theorem.
\newblock {\em The Journal of Symbolic Logic}, 38(3):359--367, 1973.

\bibitem{Kay91}
Richard Kaye.
\newblock {\em Models of {P}eano arithmetic}, volume~15 of {\em Oxford Logic
  Guides}.
\newblock Oxford Science Publications, New York, 1991.

\bibitem{KT74}
Gerog Kreisel and Gaisi Takeuti.
\newblock Formally self-referential propositions for cut free classical
  analysis and related systems.
\newblock {\em Dissertationes Mathematicae (Rozprawy Matematyczne)}, 118, 1974.

\bibitem{Kur2}
Taishi Kurahashi.
\newblock Rosser provability and the second incompleteness theorem.
\newblock In {\em Symposium on Advances in Mathematical Logic 2018
  proceedings}, 2020.
\newblock Accepted.

\bibitem{Lin03}
Per Lindstr{\"o}m.
\newblock {\em Aspects of Incompleteness}.
\newblock Number 10 in Lecture Notes in Logic. A K Peters, 2nd edition, 2003.

\bibitem{Lob55}
Martin~Hugo L{\"o}b.
\newblock Solution of a problem of {L}eon {H}enkin.
\newblock {\em The Journal of Symbolic Logic}, 20(2):115--118, 1955.

\bibitem{Mon79}
Franco Montagna.
\newblock On the formulas of {P}eano arithmetic which are provably closed under
  modus ponens.
\newblock {\em Bollettino dell'Unione Matematica Italiana}, 16(B5):196--211,
  1979.

\bibitem{Mos65}
Andrzej Mostowski.
\newblock Thirty years of foundational studies: lectures on the development of
  mathematical logic and the study of the foundations of mathematics in
  1930-1964.
\newblock In {\em Acta Philosophica Fennica}, volume~17, pages 1--180. 1965.

\bibitem{Rau10}
Wolfgang Rautenberg.
\newblock {\em A concise introduction to mathematical logic. Third edition}.
\newblock Universitext. Springer, New York, 2010.

\bibitem{Smo85}
Craig Smory{\'n}ski.
\newblock {\em Self-reference and modal logic}.
\newblock Universitext. Springer-Verlag, New York, 1985.

\bibitem{TMR53}
Alfred Tarski, Andrzej Mostowski, and Raphael~M. Robinson.
\newblock {\em Undecidable Theories}, volume~13 of {\em Studies in Logic and
  the Foundations of Mathematics}.
\newblock North-Holland Publishing, 1953.

\bibitem{Vis16}
Albert Visser.
\newblock Transductions in arithmetic.
\newblock {\em Annals of Pure and Applied Logic}, 167(3):211--234, 2016.

\bibitem{Vis}
Albert Visser.
\newblock The absorption law, or how to {K}reisel a
  {H}ilbert-{B}ernays-{L\"o}b.
\newblock arXiv: 1804.07465, 2018.

\bibitem{vB08}
Christopher von B{\"u}low.
\newblock A remark on equivalent {R}osser sentences.
\newblock {\em Annals of Pure and Applied Logic}, 151(1):62--67, 2008.

\end{thebibliography}

\end{document}